\documentclass{amsart}
\usepackage{enumerate,amssymb,mathrsfs}
\usepackage[colorlinks=true,citecolor=blue,linkcolor=blue]{hyperref}

\newcommand{\Rl}{\mathbb{R}}
\newcommand{\Cplx}{\mathbb{C}}
\newcommand{\Itgr}{\mathbb{Z}}
\newcommand{\Ntrl}{\mathbb{N}}
\newcommand{\Circ}{\mathbb{T}}

\newcommand{\Ac}{\mathcal{A}}

\newcommand{\Kc}{\mathcal{K}}
\newcommand{\Lc}{\mathcal{L}}

\newcommand{\Sc}{\mathcal{S}}

\newcommand{\Tr}{\mathrm{Tr}}

\newcommand{\Vol}{\mathrm{Vol}}

\newcommand{\diag}{\mathrm{diag}}

\def\Xint#1{\mathchoice
{\XXint\displaystyle\textstyle{#1}}%
{\XXint\textstyle\scriptstyle{#1}}%
{\XXint\scriptstyle\scriptscriptstyle{#1}}%
{\XXint\scriptscriptstyle\scriptscriptstyle{#1}}%
\!\int}
\def\XXint#1#2#3{{\setbox0=\hbox{$#1{#2#3}{\int}$ }
\vcenter{\hbox{$#2#3$ }}\kern-.6\wd0}}
\def\qint{\Xint-}

\newcommand{\re}{\mathrm{Re}}
\newcommand{\res}{\mathrm{Res}}

\newtheorem{theorem}{Theorem}[section]
\newtheorem{proposition}[theorem]{Proposition}
\newtheorem{corollary}[theorem]{Corollary}
\newtheorem{definition}[theorem]{Definition}
\newtheorem{lemma}[theorem]{Lemma}

\newtheorem{condition}[theorem]{Condition}

\theoremstyle{remark}
\newtheorem{example}[theorem]{Example}
\newtheorem{remark}[theorem]{Remark}

\title[Semiclassical Weyl law for noncommutative tori]{Semiclassical Weyl law and exact spectral asymptotics in noncommutative geometry}
\date{\today}
\author{E.~McDonald}
\author{F.~Sukochev}
\author{D.~Zanin}

\begin{document}
\maketitle{}

\begin{abstract}
    We prove a Tauberian theorem for singular values of noncommuting operators which allows us to prove exact asymptotic formulas in noncommutative geometry at
    a high degree of generality. We explain how, via the Birman--Schwinger principle, these asymptotics imply that a semiclassical Weyl law holds
    for many interesting noncommutative examples. In Connes' notation for quantized calculus, we prove that for a wide class
    of $p$-summable spectral triples $(\Ac,H,D)$ and self-adjoint $V \in \Ac$, there holds
    \[
        \lim_{h\downarrow 0} h^p\Tr(\chi_{(-\infty,0)}(h^2D^2+V)) = \qint V_-^{\frac{p}{2}}|ds|^p.
    \]
\end{abstract}

\section{Introduction}  
If $H$ is a Hilbert space, denote by $\Lc_{\infty}(H)$ the algebra of all bounded linear endomorphisms of $H$, and by $\Kc(H)$ the ideal of all compact
operators. The weak-trace class ideal $\Lc_{1,\infty}(H)$ consists of those compact operators whose singular value sequence $\mu(T)$ belongs to the weak $\ell_1$-space $\ell_{1,\infty}.$ That is,
\[
    \Lc_{1,\infty}(H) := \{T\in \Kc(H)\;:\; \sup_{t\geq 0} t\mu(t,T) < \infty\}.
\]
We recall definitions and notations relating to operator ideals and singular values in Section \ref{notation_section} below. For brevity we denote $\Lc_{\infty}=\Lc_{\infty}(H)$, $\Kc = \Kc(H),$ and so on,
when there is no possibility of confusion.

A trace on $\Lc_{1,\infty}$ is a functional $\varphi:\Lc_{1,\infty}\to \Cplx$ such that $\varphi(UBU^*) = \varphi(B)$
for all unitary operators $U$ and $B\in \Lc_{1,\infty}.$ Equivalently, $\varphi(AB) = \varphi(BA)$ for all $A \in \Lc_{\infty}$
and $B \in \Lc_{1,\infty}.$ A trace $\varphi$ is said to be normalised if $\varphi(\diag\{\frac{1}{n+1}\}_{n=0}^\infty) = 1$
and positive if $\varphi(A)\geq 0$ for all $0\leq A \in \Lc_{1,\infty}.$ The most famous traces on $\Lc_{1,\infty}$ are the Dixmier traces $\Tr_\omega$ \cite[Chapter 6]{LSZ2012}. All traces on $\Lc_{1,\infty}$ are singular, meaning
that they vanish on finite rank operators.

One of the seminal applications of singular traces is Connes' integration formula \cite[Chapter VI, Formula 2]{Connes1994}, of which one version is the following:
\begin{theorem}[Connes integration formula]\label{cif_manifold}
    Let $d > 1$ and let $(X,g)$ be a $d$-dimensional closed Riemannian manifold with volume form $\nu_g.$ If $f \in C(X)$ and if $\varphi$ is a positive normalised trace on $\Lc_{1,\infty}$, then
    \[
        \varphi((1-\Delta_g)^{-\frac{d}{4}}M_f(1-\Delta_g)^{-\frac{d}{4}}) = \frac{\Vol(S^{d-1})}{d(2\pi)^d}\int_{X} f\,d\nu_g.
    \]
    Here, $M_f$ denotes the operator on $L_2(X,\nu_g)$ of pointwise multiplication by the function $f,$ and $\Delta_g$ is the Laplace-Beltrami operator associated to $g.$
\end{theorem}
In particular, this holds when $\varphi$ is a Dixmier trace.
We caution the reader that the statement of Theorem \ref{cif_manifold} is visually different to Connes' integration formula as typically stated,
which is that
\[
    \varphi(M_f(1-\Delta_g)^{-\frac{d}{2}}) = \frac{\Vol(S^{d-1})}{d(2\pi)^d}\int_{X} f\,d\nu_g.
\]
However, the statement in Theorem \ref{cif_manifold} is equivalent. We use the theorem in ``symmetrised" form because then the operator inside
the trace is positive when $f$ is positive which has certain advantages for our purposes here, and because this form of the theorem is more easily generalised to less regular $f,$ see e.g. \cite{LSZ2020}.

A trivial consequence of Connes' integration formula is that all positive normalised traces on $\Lc_{1,\infty}$ take the same value on the operator 
$(1-\Delta_g)^{-\frac{d}{4}}M_f(1-\Delta_g)^{-\frac{d}{4}}.$ In the terminology of \cite{SSUZ2015}, we say that the operator $(1-\Delta_g)^{-\frac{d}{4}}M_f(1-\Delta_g)^{-\frac{d}{4}}$ is $\mathcal{PT}$-measurable.

It is proved in \cite[Proposition 7.2]{SSUZ2015} that a positive operator $T\in \Lc_{1,\infty}$ is $\mathcal{PT}$-measurable with $\varphi(T)=c$ for all positive normalised traces $\varphi$
if and only if
\[
    \lim_{n\to\infty}\frac{1}{n+1}\sum_{k=2^{m}}^{2^{n+m+1}-2} \mu(k,T) = c\cdot\log(2),\quad \text{ uniformly in }m\geq 0.
\]
Hence, when $f\geq 0$, Theorem \ref{cif_manifold} is equivalent to a certain statement on the asymptotic properties of the spectrum of the operator 
\[
    T = (1-\Delta_g)^{-\frac{d}{4}}M_f(1-\Delta_g)^{-\frac{d}{4}}.
\]
For this particular operator even more precise spectral asymptotics are known. For smooth positive
$f$ there holds the exact asymptotic formula,
\[
    \lim_{t\to\infty} t\mu(t,(1-\Delta_g)^{-\frac{d}{4}}M_f(1-\Delta_g)^{-\frac{d}{4}}) = \frac{\Vol(S^{d-1})}{d(2\pi)^d}\int_{X} f\,d\nu_g.
\]
Asymptotics of this type have been known for some time. See, for example, the work of Birman and Solomyak \cite{BS1971}. This is stronger than $\mathcal{PT}$-measurability and implies Theorem \ref{cif_manifold}.
For more recent work, see Rozenblum and Shargorodsky \cite{RozenblumShargorodsky2020}. In a recent preprint, Rozenblum \cite{Rozenblum2021} pointed out the connection between these Weyl-type asymptotics and Connes' integration formula.

\subsection{The integration functional on spectral triples}
In noncommutative geometry, Theorem \ref{cif_manifold} motivates the definition of the integration functional associated to an $\Lc_{p,\infty}$-summable spectral triple $(\Ac,H,D)$.
Here, $H$ is a Hilbert space, $D$ is a self-adjoint unbounded operator on $H$ and $\Ac$ is a unital sub-$*$-algebra of $\Lc_{\infty}(H).$
obeying certain conditions (see Definition \ref{spectral_triple_definition} below). The terminology is due to Connes \cite{Connes1994,Connes1995}. A spectral triple $(\Ac,H,D)$ is said to be $\Lc_{p,\infty}$-summable for $p>0$ if
\[
    (D+i)^{-1} \in \Lc_{p,\infty}(H).
\]
In the notation of Connes, $\qint$ denotes a positive normalised trace and $|ds| := (1+D^2)^{-\frac12}$ \cite{Connes1995,Connes2000}.
Given a normalised trace $\varphi$ on $\Lc_{1,\infty}$, an integration functional is defined by
\[
    \qint a |ds|^p := \varphi(a(1+D^2)^{-\frac{p}{2}}).
\] 
Unlike in the commutative case, there is no guarantee that the limit
\[
    \lim_{t\to\infty} t\mu(t,(1+D^2)^{-\frac{p}{4}}a(1+D^2)^{-\frac{p}{4}})
\]
even exists, let alone provides an additive functional on the cone of positive elements of $\Ac.$
The following theorem, which is one of the main results of this paper, is that the existence of the above limit is not particular to the commutative case, and in fact follows from a natural condition
on the spectral triple $(\Ac,H,D)$ which is easily verified in many interesting noncommutative examples.
\begin{theorem}\label{main_spectral_triple_theorem}
    Let $(\Ac,H,D)$ be an $\Lc_{p,\infty}$-summable spectral triple, where $p>2.$ Assume that $(\Ac,H,D)$ is $QC^1$ (see Definition \ref{spectral_triple_definition} below). 
    
    If, for every $0\leq a \in \Ac$ there exists $c \in \Rl$ such that the function
    \[
        z\mapsto \Tr(a^z(1+D^2)^{-\frac{z}{2}})-\frac{c}{z-p},\quad \re(z)>p
    \]
    extends continuously to the closed half-plane $\re(z)\geq p$, then for every $0\leq b \in \overline{\Ac}$ (the $C^*$-closure of $\Ac$) there exists the limit
    \[
        \lim_{t\to\infty} t\mu(t,(1+D^2)^{-\frac{p}{4}}b(1+D^2)^{-\frac{p}{4}}).
    \]
\end{theorem}
It follows from Theorem \ref{main_spectral_triple_theorem} that the operator $(1+D^2)^{-\frac{p}{4}}b(1+D^2)^{-\frac{p}{4}}$ is $\mathcal{PT}$-measurable, and hence
in Connes' notation,
\[
    \lim_{t\to\infty} t\mu(t,(1+D^2)^{-\frac{p}{4}}b(1+D^2)^{-\frac{p}{4}}) = \qint b\,|ds|^p.
\]
The choice of positive normalised trace to define $\qint$ is immaterial.
Theorem \ref{main_spectral_triple_theorem} is proved via a new mildly noncommutative variant of the Wiener-Ikehara Tauberian theorem, which is of interest in its own right.
In the cases of noncommutative tori, the conditions of the theorem are readily verified and one can even compute directly the limit, without any need for singular traces.

It is also possible to extend Theorem \ref{main_spectral_triple_theorem} to non-positive $a$ in the following way:
\begin{theorem}\label{secondary_spectral_triple_theorem}
    Let $(\Ac,H,D)$ be a spectral triple as in Theorem \ref{main_spectral_triple_theorem}. For all $a = a^* \in \overline{\Ac},$ we have the existence and coincidence of the limits
    \begin{equation*}
        \lim_{t\to\infty} t\mu\left(t,\left((1+D^2)^{-\frac{p}{4}}a(1+D^2)^{-\frac{p}{4}}\right)_+\right) = \lim_{t\to\infty} t\mu(t,(1+D^2)^{-\frac{p}{4}}a_+(1+D^2)^{-\frac{p}{4}})
    \end{equation*}
    and
    \begin{equation*}
        \lim_{t\to\infty} t\mu\left(t,\left((1+D^2)^{-\frac{p}{4}}a(1+D^2)^{-\frac{p}{4}}\right)_-\right) = \lim_{t\to\infty} t\mu(t,(1+D^2)^{-\frac{p}{4}}a_-(1+D^2)^{-\frac{p}{4}}).
    \end{equation*}
    Here, $(\cdot)_+$ and $(\cdot)_-$ denote the positive and negative parts of an operator respectively.
\end{theorem}
Employing Connes' notation, Theorem \ref{secondary_spectral_triple_theorem} implies that for all $a = a^* \in \overline{\Ac}$ we have
\[
    \lim_{t\to\infty} t\mu\left(t,\left((1+D^2)^{-\frac{p}{4}}a(1+D^2)^{-\frac{p}{4}}\right)_{\pm}\right) = \qint a_{\pm} |ds|^p.
\]

\subsection{Semiclassical Weyl asymptotics for spectral triples}
Let $(X,g)$ be a closed Riemannian $d$-manifold. For a real valued function $V \in C^\infty(X)$ and a positive parameter $h>0$, we consider the Sch\"rodinger operator
\[
    H_V := -h^2\Delta_g+M_V.
\]
Here, $\Delta_g$ is the Laplace-Beltrami operator on $X$ associated to $g.$
For $\lambda\in \Rl,$ the spectral counting function $N(\lambda,H_V)$ is the dimension of the span of the eigenvectors of $H_V$ with corresponding eigenvalue less than $\lambda.$ That is,
\[
    N(\lambda,H_V) := \Tr(\chi_{(-\infty,\lambda)}(H_V)).
\]
For further discussion of this concept and its physical relevance, see \cite{Simon1976,BS1989}.

The semiclassical Weyl law gives an asymptotic formula for $N(\lambda,H_V)$ as $h\downarrow 0.$ We denote by $t_-$ the negative part of $t\in \mathbb{R}$. That is, $t_- = \frac12(|t|-t).$
The semiclassical Weyl law for closed manifolds is as follows:
\begin{theorem}[Semiclassical Weyl law]\label{scwl_manifolds}
    Let $(X,g)$ be a closed $d$-dimensional Riemannian manifold with volume form $\nu_g,$ and let $V \in C(X)$ be real-valued.
    For all $\lambda\in \Rl$ we have
    \[
        \lim_{h\downarrow 0} h^dN(\lambda,-h^2\Delta_g+M_V) = \frac{\Vol(S^{d-1})}{d(2\pi)^d}\int_{X} (V-\lambda)_-^{\frac{d}{2}}\,d\nu_g.
    \]
\end{theorem}

Various forms of this theorem (in particular for functions on $\Rl^d$ rather than a closed manifold $X$) have a long history in the literature. See for example the work of Birman 
and Solomyak \cite{BS1971,BS1973,BS1980,BS1989}, Birman and Borzov \cite{BB1971}, Rozenblum \cite{Rozenblum1972}, Martin \cite{Martin1972} and Tamura \cite{Tamura1974}. The condition that $V$ be continuous
is far stronger than necessary, and when $d > 2$ can be weakened to $V\in L_{\frac{d}{2}}(X)$ \cite{Cwikel1977}.

It is well-known that the Birman--Schwinger principle can be used to infer semiclassical Weyl laws of the above type from asymptotic formulae
for the eigenvalues of negative order pseudodifferential operators.
We employ Theorem \ref{main_spectral_triple_theorem} to prove a semiclassical Weyl law for spectral triples:
\begin{theorem}\label{scwl_spectral_triple_thm}
    Let $(\Ac,H,D)$ be a spectral triple obeying the same conditions as in Theorem \ref{main_spectral_triple_theorem}, with $p>2,$ and let $V=V^* \in \overline{\Ac}.$ For all $\lambda\in \Rl,$ we have the existence and coincidence of the following limits:
    \[
        \lim_{h\downarrow 0} h^p N(\lambda,h^2D^2+V) = \lim_{t\to\infty} t\mu(t,(1+D^2)^{-\frac{p}{4}}(V-\lambda)_-^{\frac{p}{2}}(1+D^2)^{-\frac{p}{4}}).
    \]
\end{theorem}
The fact that these limits coincide if either exists is not really novel, being a straightforward application of the Birman--Schwinger principle as we describe below. However, what is novel is the existence of the limit on the right
hand side in a high degree of generality.

Again employing the notation of Connes, we have that
\begin{equation}\label{scwl_spectral_triple}
    \lim_{h\downarrow 0} h^pN(\lambda,h^2D^2+V) = \qint (V-\lambda)_-^{\frac{p}{2}}\,|ds|^p.
\end{equation}
Once again, the choice of positive normalized trace to define $\qint$ is irrelevant.
As with Theorem \ref{main_spectral_triple_theorem}, \eqref{scwl_spectral_triple} holds in particular for noncommutative tori.
The similarity between \eqref{scwl_spectral_triple} and Theorem \ref{scwl_manifolds} provides an interesting new interpretation of the quantised
integral $\qint a \,|ds|^p$ when $a\in \Ac$ is positive, relating the value of the integral with $N(0,h^2D^2-a^{\frac{2}{p}})$ in the limit $h\downarrow 0.$

Theorem \ref{scwl_spectral_triple_thm} holds for noncommutative tori (Corollary \ref{scwl_torus_continuous} below) and may be extended to a much wider class of $V$
using the Cwikel type estimates from \cite{MP2021} (Corollary \ref{conjecture_confirmation}). This confirms a special case of Conjecture 8.8 of \cite{MP2021}.

The authors wish to thank Prof.~Alain Connes who made many helpful comments and suggestions following a presentation on this work given by the first named author, and whose question at a 2017 conference in Fudan prompted this line of investigation. 
We also extend our gratitude to Prof.~Grigory Rozenblum, whose numerous suggestions following a careful reading of this text greatly improved the presentation. Discussions and collaboration with Rapha\"el Ponge are also gratefully acknowledged.

\section{Preliminary material on singular values}\label{notation_section}
Recall that we denote by $\Kc = \Kc(H)$ the $C^*$-subalgebra of all compact endomorphisms of a Hilbert space $H,$
 and $\mathcal{L}_\infty = \mathcal{L}_\infty(H)$ is the algebra of all bounded linear endomorphisms of $H$. The operator
norm on $\mathcal{L}_\infty$ is $\|\cdot\|.$
If $T\in \Kc$ is a compact operator on $H$, the sinuglar value function $t\mapsto \mu(t,T)$
may be defined as
\[
    \mu(t,T) := \inf\{\|T-R\|\;:\;\mathrm{rank}(T)\leq t\},\quad t \geq 0.
\]
We denote $\mu(T)$ for the sequence $\{\mu(n,T)\}_{n\geq 0}$. Equivalently, $\mu(T)$ is the sequence of eigenvalues of the absolute value $|T|$ arranged in non-increasing order with multiplicities.
The Schatten ideal $\Lc_p=\Lc_p(H)$ is the ideal of $T \in \Kc$ such that
\[
    \|T\|_p := \|\mu(T)\|_{\ell_p} = \left(\sum_{n=0}^\infty \mu(n,T)^p\right)^{\frac{1}{p}}.
\]    
The ideal $\Lc_{1,\infty}$ is defined as the quasi-Banach space of compact operators $T$ with $\mu(t,T) = O(t^{-1})$ as $t\to\infty.$ To be precise, for $p>0$ we define
\[
    \Lc_{p,\infty} := \{T \in \Kc\;:\;\|T\|_{p,\infty} := \sup_{t\geq 0} t^{\frac{1}{p}}\mu(t,T) < \infty\}.
\]
This is a quasi-Banach ideal of $\Lc_{\infty}.$ The set of finite rank operators is not dense in $\Lc_{p,\infty}.$ Instead, the closure
of the set of finite rank operators in the $\Lc_{p,\infty}$-quasinorm is denoted $(\Lc_{p,\infty})_0$, and coincides
with the set of $T\in \Kc$ such that $\lim_{t\to\infty} t^{\frac{1}{p}}\mu(t,T) =0.$
Similarly, the ideal $\Lc_{p,1}$ is defined as
\[
    \Lc_{p,1} := \{T\in \Kc(H)\;:\;\|T\|_{p,1} := \sum_{n=0}^\infty (n+1)^{\frac{1}{p}-1}\mu(n,T) < \infty.
\]
We have the H\"older-type inequality
\begin{equation}\label{weird_holder}
    \|TS\|_{1} \leq \|T\|_{p,\infty}\|S\|_{q,1},\quad \frac{1}{p}+\frac{1}{q} = 1.
\end{equation}

\subsection{Outline of the proofs}\label{overview_section}
The main workhorse for our arguments is a mildly noncommutative version of the Wiener-Ikehara theorem, which we explain briefly now before
going into technical details in the next section.

In the language of compact operators, the Wiener-Ikehara theorem \cite[Chapter III, Theorem 5.1]{Korevaar-tauberian-2004} can be stated as follows.
If $0 \leq V \in \Kc$ and there exists $c \in \Rl$ and $p > 0$ such that the function
\begin{equation*}
    F_V(z) := \Tr(V^z)-\frac{c}{z-p},\quad \re(z)>p
\end{equation*}
is well-defined\footnote{In particular, $V^r \in \mathcal{L}_1$ for every $r>p.$} and admits a continuous extension to the closed half-plane $\re(z)\geq p$, it follows that there exists the limit
\[
    \lim_{t\to\infty} t\mu(t,V^p) = \frac{c}{p}.
\]

We prove a new Tauberian theorem involving two positive bounded operators $A$ and $B$ which do not necessarily commute and which satisfy the following two conditions:
\begin{enumerate}[{\rm (i)}]
    \item{} $B \in \Lc_{p,\infty},$ where $p>2,$
    \item{} $[A^{\frac12},B] \in \Lc_{\frac{p}{2},\infty}.$
\end{enumerate}
(Later we will refer to these assumptions as Condition \ref{compact conditions for analyticity}).
If there exists $c\in \Rl$ such that the function
\[
    F_{A,B}(z) := \Tr(A^zB^{z})-\frac{c}{z-p},\quad \re(z)>p.
\]
admits a continuous extension to the closed half-plane $\re(z)\geq p$, then there exists the limit
\[
    \lim_{t\to\infty} t\mu(t,B^{\frac{p}{2}}A^pB^{\frac{p}{2}}) = \frac{c}{p}.
\]
(see Corollary \ref{tauberian_corollary}).
In the case that $A$ and $B$ commute, this recovers the Wiener-Ikehara theorem, save for the extra \emph{a priori} assumptions that $B\in \Lc_{p,\infty}$ and the restriction $p>2,$ which are redundant in
the Wiener-Ikehara theorem. The restriction $p>2$ is made to simplify the exposition and is not likely to be essential.

This Tauberian theorem is the basis for the proof of Theorem \ref{main_spectral_triple_theorem}, where we take $B = (1+D^2)^{-\frac12}$ and $A = 
a^{2},$ and the assumptions
on $A$ and $B$ are restated in terms of the spectral triple.

\section{A noncommutative Tauberian theorem}
We will work with operators $A$ and $B$ and a parameter $p$ which obey the following:
\begin{condition}\label{compact conditions for analyticity} 
    Let $p>2$ and let $0\leq A,B\in\mathcal{L}_{\infty}$ satisfy the conditions
    \begin{enumerate}[{\rm (i)}]
        \item\label{ccond1} $B\in\mathcal{L}_{p,\infty}$
        \item\label{ccond2} $[B,A^{\frac12}]\in\mathcal{L}_{\frac{p}{2},\infty}.$
    \end{enumerate}
\end{condition}

This section is primarily devoted to the proof of the following assertion:
\begin{theorem}\label{main_analytic_theorem}
    If $A,B$ and $p$ satisfy Condition \ref{compact conditions for analyticity}, then
    \[
        f_{A,B}(z) := \Tr((A^{\frac12}BA^{\frac12})^z-B^zA^z),\quad \re(z)>p
    \]
    extends continuously to the closed half-plane $\{z\in \Cplx\;:\;\re(z)\geq p\}.$
    More specifically, there exists $n\in \Ntrl$ and an absolute constant $c_{{\rm abs}}$ such that
    \[
        |f_{A,B}(z_1)-f_{A,B}(z_2)|\leq c_{{\rm abs}}|z_1-z_2|\cdot (1+|z_1|)^n\cdot (1+|z_2|)^n,\quad p < \re(z_1), \re(z_2) \leq p+1.
    \]
\end{theorem}
The claim that $f_{A,B}$ has continuous extension is a special case of \cite[Theorem 5.4.2]{SZ-asterisque}, which states under more general conditions that $f_{A,B}$ has analytic continuation to the half-plane $\re(z)>p-1$, but the weaker statement in \eqref{main_analytic_theorem} has some independent interest and can be proved a little more easily, and thus we give a proof in the following section. 

Theorem \ref{main_analytic_theorem} implies the following Tauberian theorem:
\begin{theorem}\label{main_tauberian_theorem}
    Let $A, B$ and $p$ obey Condition \ref{compact conditions for analyticity}. 
    If there exists $c\in \Rl$ such that the function
    \[
        F_{A,B}(z) := \Tr(B^zA^z)-\frac{c}{z-p},\quad \re(z)>p
    \]
    has a continuous extension to the closed half-plane $\{z\in \Cplx\;:\;\re(z)\geq p\},$ then there exists the limit
    \[
        \lim_{t\to \infty} t\mu(t,(A^{\frac{1}{2}}BA^{\frac{1}{2}})^p) = \frac{c}{p}.
    \]
\end{theorem}
\begin{proof}
    Applying Theorem \ref{main_analytic_theorem} and the assumption, the function
    \[
        f_{A,B}(z)+F_{A,B}(z) = \Tr((A^{\frac12}BA^{\frac12})^z) - \frac{c}{z-p},\quad \re(z)\geq p.
    \]  
    admits a continuous extension to the closed half plane $\re(z)\geq p.$
    By the Wiener--Ikehara theorem, it follows that
    \[
        \lim_{t\to\infty} t\mu(t,(A^{\frac{1}{2}}BA^{\frac12})^p) = \frac{c}{p}.
    \]
\end{proof}
The statement in Theorem \ref{main_analytic_theorem} is reminiscent of some earlier zeta function estimates in noncommutative geometry, such as those in \cite{CGRS1} and \cite{CPS2003}, but the fundamental difference is that we allow the variable $z$ to be complex which allows us to strengthen the conclusion of the theorem. This necessitates the use of new tools such as Theorem \ref{main_analytic_theorem}, which was unavailable in \cite{CGRS1,CPS2003}.

The following lemma is a modification of a result in \cite{HSZ}.
\begin{lemma}\label{HSZ_modified_inequality}
    Let $p,q > 0,$ and $0\leq X \in \Lc_{q,\infty}.$ If $Y \in \Lc_{\infty}$ is such that $[X,Y] \in (\Lc_{q,\infty})_0,$
    then $[X^p,Y] \in (\Lc_{\frac{q}{p},\infty})_0.$
\end{lemma}
\begin{proof}
    Initially suppose that $p\in \Ntrl.$ In this case we have
    \[
        [X^p,Y] = \sum_{k=0}^{p-1} X^{k}[X,Y]X^{p-1-k} \in \sum_{k=0}^{p-1} \Lc_{\frac{q}{k},\infty}\cdot (\Lc_{q,\infty})_0\cdot \Lc_{\frac{q}{p-1-k},\infty}.
    \]
    This is contained in $(\Lc_{\frac{q}{p},\infty})_0$, by H\"older's inequality.
    
    For the general case, we apply \cite[Corollary 7.1]{HSZ} which implies that if $0 < \theta < 1$ then
    \[
        [X,Y] \in (\Lc_{q,\infty})_0 \Longrightarrow [X^\theta,Y] \in (\Lc_{\frac{q}{\theta},\infty})_0.
    \]
    If $p = n+\theta,$ where $n\in \Ntrl$ and $\theta \in (0,1),$ then by the Leibniz rule we have
    \[
        [X^p,Y] = X^\theta[X^n,Y] + X^n[X^\theta,Y] \in \Lc_{\frac{q}{\theta},\infty}\cdot (\Lc_{\frac{q}{n},\infty})_0+\Lc_{\frac{q}{n},\infty}\cdot (\Lc_{\frac{q}{\theta},\infty})_0.
    \]
    By H\"older's inequality, it follows that $[X^p,Y] \in (\Lc_{\frac{q}{p},\infty})_0.$
\end{proof}

In the next lemma we apply the main result of \cite{PS2011}, which implies that if $1<q<\infty$ and $f$ is a Lipschitz continuous function on $\Rl$,
then there exists a constant $C_{q}$ such that
\begin{equation}\label{ps_acta_inequality}
    \|[f(X),Y]\|_{q,\infty} \leq C_{q}\|f'\|_{\infty}\|[X,Y]\|_{q,\infty}
\end{equation}
for all self-adjoint bounded linear operators $X$ and $Y$ with $[X,Y] \in \Lc_{q,\infty}.$ There are also the implications
\begin{align*}
    [X,Y] \in \Lc_{q,\infty}&\Longrightarrow [f(X),Y] \in \Lc_{q,\infty},\\
    [X,Y] \in (\Lc_{q,\infty})_0&\Longrightarrow [f(X),Y] \in (\Lc_{q,\infty})_0.
\end{align*}

\begin{lemma}\label{from_csz_lemma}
Let $A\geq0,$ $B\geq0$ and $p>1$ are such that
\begin{enumerate}[{\rm (i)}]
\item $B\in\mathcal{L}_{p,\infty};$
\item $[A^{\frac12},B]\in(\mathcal{L}_{p,\infty})_0.$
\end{enumerate}
It follows that
    \[
        A^{\frac{p}{2}}B^pA^{\frac{p}{2}}-(A^{\frac{1}{2}}BA^{\frac12})^p\in (\Lc_{1,\infty})_0.
    \]
\end{lemma}
\begin{proof}
    It was proved in \cite[Lemma 5.3]{CSZ} that under these conditions
    \[
        B^pA^p - (A^{\frac12}BA^{\frac12})^p \in (\Lc_{1,\infty})_0.
    \]
    Thus, we only need to show that
    \begin{equation}\label{from_csz}
        B^pA^p-A^{\frac{p}{2}}B^pA^{\frac{p}{2}} = [B^p,A^{\frac{p}{2}}]A^{\frac{p}{2}} \in (\Lc_{1,\infty})_0.
    \end{equation}
  
    The function $f(t) := t^p$ is Lipschitz on the spectrum of $B$ (which is necessarily a bounded subset of $\Rl$). The second assumption and $p>1$ implies
    \[
        [B,A^{\frac{p}{2}}]\in (\Lc_{p,\infty})_0.
    \]
    Applying Lemma \ref{HSZ_modified_inequality} with $X = B$, $Y = A^{\frac{p}{2}}$ and $q=p$, it follows that
    \[
        [B^{p},A^{\frac{p}{2}}] \in (\Lc_{1,\infty})_0.
    \]
    This yields \eqref{from_csz}, which completes the proof.        
\end{proof}

The following lemma is due to Birman and Solomyak \cite[Lemma 4.1]{BS1971}.
\begin{lemma}\label{bs_perturbation_lemma}
    Let $T \in \Lc_{p,\infty}$ be an operator such that for every $\varepsilon>0$ there exists a decomposition
    \[
        T = T_{\varepsilon}'+T_{\varepsilon}''
    \]  
    such that there exists the limit
    \[
        \lim_{t\to\infty} t^{\frac{1}{p}}\mu(t,T'_{\varepsilon}) =: c(\varepsilon)
    \]
    and
    \[
        \limsup_{t\to\infty} t^{\frac{1}{p}}\mu(t,T_{\varepsilon}'') \leq \varepsilon.
    \]
    Then we have the existence and equality of the following limits:
    \[
        \lim_{t\to\infty} t^{\frac{1}{p}}\mu(t,T) = \lim_{\varepsilon\to 0}c(\varepsilon).
    \]
\end{lemma}

It is helpful to have the following consequence of Lemma \ref{bs_perturbation_lemma}, which essentially states that the set of $X \in \Lc_{p,\infty}$ such that $\{n^{\frac{1}{p}}\mu(n,X)\}_{n\geq 0}$ is convergent
is closed.
\begin{lemma}\label{bs_perturbation_lemma_sequential}
Let $p>0.$ Suppose $\{X_k\}_{k\geq0}\subset\mathcal{L}_{p,\infty}$ are such that, for every $k\geq0,$ there exists a limit
\[
    \lim_{t\to\infty}t^{\frac{1}{p}}\mu(t,X_k)=a_k.
\]
Suppose $X_k\to X$ in $\mathcal{L}_{p,\infty}.$ It follows that the sequences $\{n^{\frac{1}{p}}\mu(n,X)\}_{n\geq 0}$ and $\{a_k\}_{k\geq0}$ both converge, and moreover to the same limit. That is,
\[
    \lim_{t\to\infty}t^{\frac{1}{p}}\mu(t,X)=\lim_{k\to\infty}a_k.
\]
\end{lemma}

The following is an easy consequence.
\begin{corollary}\label{simplified_bs_perturbation_lemma} Let $p>0.$ If $X,Y\in\mathcal{L}_{p,\infty}$ are such that $X-Y\in(\mathcal{L}_{p,\infty})_0,$ then
$$\lim_{t\to\infty}t\mu(t,X)^p=\lim_{t\to\infty}t\mu(t,Y)^p$$
provided that either limit exists.
\end{corollary}

We now complete the proof of the Tauberian theorem in the form stated in Section \ref{overview_section}.
\begin{theorem}\label{tauberian_corollary}
    Under the conditions of Theorem \ref{main_tauberian_theorem}, we have
    \[
        \lim_{t\to \infty} t\mu(t,B^{\frac{p}{2}}A^pB^{\frac{p}{2}}) = \frac{c}{p}.
    \]
\end{theorem}
\begin{proof}
    Combining Theorem \ref{main_tauberian_theorem} with Lemma \ref{from_csz_lemma} yields
    \[
        \lim_{t\to\infty} t\mu(t,(A^{\frac{1}{2}}BA^{\frac{1}{2}})^p) = \frac{c}{p},\quad A^{\frac{p}{2}}B^pA^{\frac{p}{2}} - (A^{\frac{1}{2}}BA^{\frac{1}{2}})^p \in (\Lc_{1,\infty})_0.
    \]
    It follows from Corollary \ref{simplified_bs_perturbation_lemma} that
    \[
        \lim_{t\to\infty} t\mu(t,A^{\frac{p}{2}}B^pA^{\frac{p}{2}}) = \frac{c}{p}.
    \]
    The operators $A^{\frac{p}{2}}B^pA^{\frac{p}{2}}$ and $B^{\frac{p}{2}}A^pB^{\frac{p}{2}}$ have identical nonzero spectra, which implies that
    \[
        \mu(A^{\frac{p}{2}}B^pA^{\frac{p}{2}}) = \mu(B^{\frac{p}{2}}A^pB^{\frac{p}{2}}).
    \]
    This completes the proof.
\end{proof}

\subsection{Integral representation for differences of complex powers}
We now proceed with the proof of the Theorem \ref{main_analytic_theorem}. The proof is based on the following representation of the difference
\[
    (A^{\frac12}BA^{\frac12})^z-B^zA^z
\]
which was first stated in \cite[Lemma 5.2]{CSZ}, and later strengthened in \cite[Theorem 5.2.1]{SZ-asterisque}.
\begin{theorem}\label{csz key lemma} 
Let $A$ and $B$ be bounded, positive operators on $H$, and let $z \in \mathbb{C}$ with $\re(z) > 1$. Let $Y := A^{\frac12}BA^{\frac12}.$ We define the mapping $T_z:\mathbb{R}\to \mathcal{L}_\infty$ by,
\begin{align*}
T_z(0) &:= B^{z-1}[BA^{\frac{1}{2}},A^{z-\frac{1}{2}}]+[BA^{\frac{1}{2}},A^{\frac{1}{2}}]Y^{z-1},\\
T_z(s) &:= B^{z-1+is}[BA^{\frac{1}{2}},A^{z-\frac{1}{2}+is}]Y^{-is}+B^{is}[BA^{\frac{1}{2}},A^{\frac{1}{2}+is}]Y^{z-1-is},\quad s \neq 0.
\end{align*}
We also define the function $g_z:\mathbb{R}\to \mathbb{C}$ by:
\begin{align*}
g_z(0) &:= 1-\frac{z}{2},\\
g_z(t) &:= 1-\frac{e^{\frac{z}{2}t}-e^{-\frac{z}{2}t}}{(e^{\frac{t}{2}}-e^{-\frac{t}{2}})(e^{\left(\frac{z-1}{2}\right)t}+e^{-\left(\frac{z-1}{2}\right)t})}.
\end{align*} 
Then:
\begin{enumerate}[{\rm (i)}]
\item{} The mapping $T_z:\mathbb{R}\to \mathcal{L}_\infty$ is continuous in the weak operator topology.
\item{} We have:
\begin{equation*}
B^zA^z-(A^{\frac{1}{2}}BA^{\frac{1}{2}})^z = T_z(0)-\int_{\mathbb{R}} T_z(s)\widehat{g}_z(s)\,ds. 
\end{equation*}
\end{enumerate}
\end{theorem}

For $\re(z) > 1$, the function $g_z$ belongs to the Schwartz class $\Sc(\Rl)$ \cite[Remark 5.2.2]{SZ-asterisque}, and hence so is the Fourier transform $\widehat{g}_z,$
which we normalise as
\[
    \widehat{g}_z(\xi) = \int_{-\infty}^\infty e^{-i\xi t}g_z(t)\,dt.
\]

\begin{remark}\label{normalisation remark} In what follows, we assume without loss of generality that $\|B\|_{p,\infty}=1$ and that $\|A\|=1.$ In particular, this implies $\|Y\|_{p,\infty}\leq 1.$
\end{remark}

We will now proceed to estimate the individual summands in $T_z(s)$ from Theorem \ref{csz key lemma}.

\begin{lemma}\label{first psacta lemma} 
Let $p>2$ and let $A,B\in\mathcal{L}_{\infty}$ satisfy Condition \ref{compact conditions for analyticity} and normalised as in Remark \ref{normalisation remark}. We have
$$\|[BA^{\frac{1}{2}},A^{z-\frac{1}{2}+is}]\|_{\frac{p}{2},\infty}\leq c_p(1+|z|+|s|)\|[BA^{\frac{1}{2}},A^{\frac12}]\|_{\frac{p}{2},\infty},\quad\re(z)\geq 1,$$
$$\|[BA^{\frac{1}{2}},A^{\frac{1}{2}+is}]\|_{\frac{p}{2},\infty}\leq c_p(1+|s|)\|[BA^{\frac{1}{2}},A^{\frac12}]\|_{\frac{p}{2},\infty}.$$
\end{lemma}
\begin{proof} 
We only prove the first inequality as the proof of the second one is identical.
Fix $z$ with $\re(z)\geq 1,$ and define the function
$$\phi(t):=
\begin{cases}
t^{2z-1+2is},& t\in(0,1)\\
0,& t\leq 0\\
1,& t\geq 1
\end{cases}
$$
By the normalisation assumption, we have $\|A\|\leq 1.$ Therefore,
$$A^{z-\frac12+is}=\phi(A^{\frac12}).$$
Applying \eqref{ps_acta_inequality} to the Lipschitz function $\phi$, we have
\begin{align*}
    \|[BA^{\frac{1}{2}},A^{z-\frac{1}{2}+is}]\|_{\frac{p}{2},\infty} &= \|[BA^{\frac{1}{2}},\phi(A^{\frac12})]\|_{\frac{p}{2},\infty}\\
                                                                     &\leq c_p\|\phi'\|_{\infty}\|[BA^{\frac{1}{2}},A^{\frac12}]\|_{\frac{p}{2},\infty}.
\end{align*}
Due to the obvious inequality,
$$\|\phi'\|_{\infty}=|z-\frac12+is|\leq 1+|z|+|s|$$
this completes the proof.
\end{proof}

\begin{lemma}\label{tz boundedness lemma} 
Let $p>2$ and let $A$ and $B$ satisfy Condition \ref{compact conditions for analyticity} and normalised as in Remark \ref{normalisation remark}, and let $T_z$ be the mapping from Theorem \ref{csz key lemma}. There exists a constant $c_{p,A,B}$ such that for all $s \in \Rl,$
$$\|T_z(s)\|_1\leq  c_{p,A,B}(1+|z|)(1+|s|)\|[BA^{\frac{1}{2}},A^{\frac12}]\|_{\frac{p}{2},\infty},\quad \re(z)\geq p.$$
\end{lemma}
\begin{proof} 
By the triangle inequality for $\Lc_1$, we have
\begin{align*}
    \|T_z(s)\|_1&\leq\|B^{z-1+is}[BA^{\frac{1}{2}},A^{z-\frac{1}{2}+is}]Y^{-is}\|_1+\|B^{is}[BA^{\frac{1}{2}},A^{\frac{1}{2}+is}]Y^{z-1-is}\|_1\\
                &\leq\|B^{z-1}[BA^{\frac{1}{2}},A^{z-\frac{1}{2}+is}]\|_1+\|[BA^{\frac{1}{2}},A^{\frac{1}{2}+is}]Y^{z-1}\|_1,\quad s\neq 0.
\end{align*}
Applying the H\"older-type inequality \eqref{weird_holder}, we infer
\begin{align*}
    \|T_z(s)\|_1 &\leq\|B^{z-1}\|_{\frac{p}{p-2},1}\|[BA^{\frac{1}{2}},A^{z-\frac{1}{2}+is}]\|_{\frac{p}{2},\infty}\\
                 &\quad +\|[BA^{\frac{1}{2}},A^{\frac{1}{2}+is}]\|_{\frac{p}{2},\infty}\|Y^{z-1}\|_{\frac{p}{p-2},1},\quad s\neq 0.
\end{align*}
By the normalisation, we have $\|B\|_{p,\infty},\|Y\|_{p,\infty}\leq 1.$ This allows us to write
\begin{align*}
    \|B^{z-1}\|_{\frac{p}{p-2},1} &\leq \Big\|\Big\{k^{\frac{1-z}{p}}\Big\}_{k=1}^{\infty}\Big\|_{\frac{p}{p-2},1}\leq\Big\|\Big\{k^{\frac{1-p}{p}}\Big\}_{k=1}^{\infty}\Big\|_{\frac{p}{p-2},1}=:c_p',\quad\re(z)\geq p,\\
\|Y^{z-1}\|_{\frac{p}{p-2},1} &\leq \Big\|\Big\{k^{\frac{1-z}{p}}\Big\}_{k=1}^{\infty}\Big\|_{\frac{p}{p-2},1}\leq\Big\|\Big\{k^{\frac{1-p}{p}}\Big\}_{k=1}^{\infty}\Big\|_{\frac{p}{p-2},1} = c_p',\quad\re(z)\geq p,
\end{align*}
Thus,  for $s\neq 0$
$$\|T_z(s)\|_1\leq  c_p'(\|[BA^{\frac{1}{2}},A^{z-\frac{1}{2}+is}]\|_{\frac{p}{2},\infty}+\|[BA^{\frac{1}{2}},A^{\frac{1}{2}+is}]\|_{\frac{p}{2},\infty}).$$
Using Lemma \ref{first psacta lemma}, we write
$$\|[BA^{\frac{1}{2}},A^{z-\frac{1}{2}+is}]\|_{\frac{p}{2},\infty}+\|[BA^{\frac{1}{2}},A^{\frac{1}{2}+is}]\|_{\frac{p}{2},\infty}\leq c_p''(1+|z|+|s|)\|[BA^{\frac{1}{2}},A^{\frac12}]\|_{\frac{p}{2},\infty}.$$
Thus,
$$\|T_z(s)\|_1\leq c_p'c_p''(1+|z|+|s|)\|[BA^{\frac{1}{2}},A^{\frac12}]\|_{\frac{p}{2},\infty}+c_p'(1+|s|)\|[BA^{\frac{1}{2}},A^{\frac12}]\|_{\frac{p}{2},\infty}.$$
Since $(1+|z|+|s|)\leq (1+|z|)(1+|s|)$, the required assertion for $s\neq 0$ follows. The $s=0$ case proved similarly.
\end{proof}

\begin{lemma}\label{continuity phi compute lemma} Let $\re(z_1),\re(z_2)>1,$ and set
$$\phi_{z_1,z_2}(t)=t^{z_1}-t^{z_2},\quad t\in (0,1)$$
and $\phi_{z_1,z_2}(0) = \phi_{z_1,z_2}(1) = 0.$
We have
$$\|\phi_{z_1,z_2}\|_{\infty}\leq\frac{|z_1-z_2|}{\min\{\re(z_1)-1,\re(z_2)-1\}},$$
$$\|\phi_{z_1,z_2}'\|_{\infty}\leq |z_1-z_2|\cdot \max\{|z_1|,|z_2|\}\cdot\Big(1+\frac1{\min\{\re(z_1)-1,\re(z_2)-1\}}\Big).$$
\end{lemma}
\begin{proof} 
We first prove the upper bound for $\|\phi_{z_1,z_2}\|_{\infty}.$
By definition,
\[
    \|\phi_{z_1,z_2}\|_{\infty}=\sup_{0\leq t\leq 1}|t^{z_1}-t^{z_2}|.
\]
Now,
$$\frac{t^{z_1}-t^{z_2}}{z_1-z_2}=\int_0^1t^{\lambda z_1+(1-\lambda)z_2-1}\cdot \log(t)d\lambda.$$
Assume without loss of generality that $\re(z_1)< \re(z_2).$ Then,
\begin{align*}
    \frac{\|\phi_{z_1,z_2}\|_{\infty}}{|z_1-z_2|}&\leq \sup_{\lambda\in[0,1]}\sup_{0\leq t\leq 1}|t^{\lambda z_1+(1-\lambda)z_2-1}\cdot \log(t)|\\
                                                 &=\sup_{\mu\in[\re(z_1),\re(z_2)]}\sup_{0\leq t\leq 1}t^{\mu-1}\log(\frac1t)\\
                                                 &=\sup_{\mu\in[\re(z_1),\re(z_2)]}\frac1{\mu-1}\sup_{0\leq t\leq 1}t^{\mu-1}\log(\frac1{t^{\mu-1}})\\
                                                 &=\sup_{\mu\in[\re(z_1),\re(z_2)]}\frac1{\mu-1}\cdot \sup_{0\leq t\leq 1}t\log(\frac1t)\leq \sup_{\mu\in[\re(z_1),\re(z_2)]}\frac1{\mu-1}.
\end{align*}
This proves the first inequality.

Next, we prove the upper bound for $\|\phi'_{z_1,z_2}\|_{\infty}.$ By definition,
$$\|\phi'_{z_1,z_2}\|_{\infty}=\sup_{0\leq t\leq 1}|z_1t^{z_1-1}-z_2t^{z_2-1}|.$$
Now,
\begin{equation*}
    \frac{z_1t^{z_1-1}-z_2t^{z_2-1}}{z_1-z_2} =\int_0^1t^{\lambda z_1+(1-\lambda)z_2-1}\cdot (1+(\lambda z_1+(1-\lambda)z_2)\log(t))d\lambda.
\end{equation*}
Thus,
\begin{align*}
    \frac{\|\phi'_{z_1,z_2}\|_{\infty}}{|z_1-z_2|} &\leq \sup_{\lambda\in[0,1]}\sup_{0\leq t\leq 1}|t^{\lambda z_1+(1-\lambda)z_2-1}\cdot (1+(\lambda z_1+(1-\lambda)z_2)\log(t))|\\
                                                   &\leq \max\{|z_1|,|z_2|\}\cdot \sup_{\mu\in[\re(z_1),\re(z_2)]}\sup_{0\leq t\leq 1}t^{\mu-1}\cdot (1+\log(\frac1t))\\
                                                   &\leq \max\{|z_1|,|z_2|\}\cdot\Big(1+\sup_{\mu\in[\re(z_1),\re(z_2)]}\sup_{0\leq t\leq 1}t^{\mu-1}\log(\frac1t)\Big)\\
                                                   &=\max\{|z_1|,|z_2|\}\cdot\Big(1+\sup_{\mu\in[\re(z_1),\re(z_2)]}\frac1{\mu-1}\sup_{0\leq t\leq 1}t^{\mu-1}\log(\frac1{t^{\mu-1}})\Big)\\
                                                   &=\max\{|z_1|,|z_2|\}\cdot\Big(1+\sup_{\mu\in[\re(z_1),\re(z_2)]}\frac1{\mu-1}\cdot \sup_{0\leq t\leq 1}t\log(\frac1t)\Big)\\
                                                   &\leq \max\{|z_1|,|z_2|\}\cdot\Big(1+\sup_{\mu\in[\re(z_1),\re(z_2)]}\frac1{\mu-1}\Big).
\end{align*}
This completes the proof of the second inequality.
\end{proof}

\begin{lemma}\label{second psacta lemma} 
Let $p>2$ and let $A$ and $B$ be positive operators satisfying Condition \ref{compact conditions for analyticity} and normalised as in Remark \ref{normalisation remark}. There exists $c_p>0$ such that
$$\Big\|[BA^{\frac12},A^{z_1-1}-A^{z_2-1}]\Big\|_{\frac{p}{2},\infty}\leq c_p|z_1-z_2|\cdot\max\{1,|z_1|,|z_2|\}\cdot \|[BA^{\frac{1}{2}},A^{\frac12}]\|_{\frac{p}{2},\infty},$$
$$\Big\|Y^{z_1-1}-Y^{z_2-1}\Big\|_{\frac{p}{p-2},1} \leq c_p|z_1-z_2|,\quad \Big\|B^{z_1-1}-B^{z_2-1}\Big\|_{\frac{p}{p-2},1} \leq c_p|z_1-z_2|,$$
whenever $\re(z_1),\re(z_2)\geq p.$
\end{lemma}
\begin{proof} 
We apply \eqref{ps_acta_inequality}, similarly to the proof of Lemma \ref{first psacta lemma}.  By the normalisation, we have $\|A\|\leq 1.$ Therefore,
$$A^{z_1-1}-A^{z_2-1}=\phi_{2z_1-2,2z_2-2}(A^{\frac12}).$$
By \eqref{ps_acta_inequality} with $f = \phi_{2z_1-2,2z_2-2}$, we have
\begin{align*}
    \Big\|[BA^{\frac12},A^{z_1-1}-A^{z_2-1}]\Big\|_{\frac{p}{2},\infty} &= \Big\|[BA^{\frac12},\phi_{2z_1-2,2z_2-2}(A^{\frac12})]\Big\|_{\frac{p}{2},\infty}\\
                                                                        &\leq c_p\|\phi_{2z_1-2,2z_2-2}'\|_{\infty}\Big\|[BA^{\frac12},A^{\frac12}]\Big\|_{\frac{p}{2},\infty}.
\end{align*}
Note that $\re(2z_1-2),\re(2z_2-2)\geq 2p-2\geq 2.$ By Lemma \ref{continuity phi compute lemma}, we have
\[
    \|\phi_{2z_1-2,2z_2-2}'\|_{\infty}\leq 2|(2z_1-2)-(2z_2-2)|\cdot\max\{|2z_1-2|,|2z_2-2|\}.
\]
This completes the proof of the first assertion

To see the second inequality, we write
$$Y^{z_1-1}-Y^{z_2-1}=Y^{p+\frac14}\cdot \phi_{2z_1-\frac52,2z_2-\frac52}(Y^{\frac12}).$$
By H\"older inequality, we have
$$\|Y^{z_1-1}-Y^{z_2-1}\|_{\frac{p}{p-2},1} \leq\|Y^{p+\frac14}\|_{\frac{p}{p-2},1} \cdot\|\phi_{2z_1-\frac52,2z_2-\frac52}(Y^{\frac12})\|.$$
Since $Y\in\mathcal{L}_{p,\infty},$ it follows that
$$c_p\stackrel{def}{=}\|Y^{p+\frac12}\|_{\frac{p}{p-2},1}<\infty.$$
By the normalisation, we have $\|Y\|\leq 1.$ Thus, 
$$\|\phi_{2z_1-\frac52,2z_2-\frac52}(Y^{\frac12})\|\leq \|\phi_{2z_1-\frac52,2z_2-\frac52}\|_{\infty}.$$
Note that $\re(2z_1-\frac52),\re(2z_2-\frac52)\geq 2p-\frac52\geq\frac32.$ By Lemma \ref{continuity phi compute lemma}, we have
$$\|\phi_{2z_1-\frac52,2z_2-\frac52}\|_{\infty}\leq 3|(2z_1-\frac52)-(2z_2-\frac52)|=6|z_1-z_2|.$$
This completes the proof of the second inequality. 

The third inequality is proved in an identical manner.
\end{proof}

\begin{lemma}\label{tz continuity lemma} Let $p>2$ and let $A$ and $B$ be positive operators satisfying Condition \ref{compact conditions for analyticity} and normalised as in Remark \ref{normalisation remark}. There exists a constant $c_p$ such that for all $s \in \Rl$ and
$z_1,z_2$ satisfying $\re(z_1),\re(z_2)\geq p,$
we have
\begin{align*}
  \|T_{z_1}(s)-&T_{z_2}(s)\|_1\leq c_p(1+|z_1|+|z_2|+|s|)|z_1-z_2|\Big\|[BA^{\frac{1}{2}},A^{\frac{1}{2}}]\Big\|_{\frac{p}{2},\infty}.
\end{align*}
\end{lemma}
\begin{proof} 
We deal with $s\neq 0$, with the $s=0$ case being similar.
We begin by rewriting $T_{z_1}(s)-T_{z_2}(s)$ as
\begin{align*}
    T_{z_1}(s)-T_{z_2}(s)&=B^{is}[BA^{\frac{1}{2}},A^{\frac{1}{2}+is}]\cdot\big(Y^{z_1-1-is}-Y^{z_2-1-is}\big)\\
                         &+\big(B^{z_1-1+is}-B^{z_2-1+is}\big)[BA^{\frac{1}{2}},A^{z_1-\frac{1}{2}+is}]Y^{-is}\\
                         &+B^{z_2-1+is}\big([BA^{\frac{1}{2}},A^{z_1-\frac{1}{2}+is}-A^{z_2-\frac12+is}]\big)Y^{-is}.
\end{align*}
By the triangle inequality in $\Lc_1,$ we have
\begin{align*}
\Big\|T_{z_1}(s)-T_{z_2}(s)\Big\|_1&\leq \Big\|[BA^{\frac{1}{2}},A^{\frac{1}{2}+is}]\cdot\big(Y^{z_1-1}-Y^{z_2-1}\big)\Big\|_1\\
                                   &+\Big\|\big(B^{z_1-1}-B^{z_2-1}\big)\cdot [BA^{\frac{1}{2}},A^{z_1-\frac{1}{2}+is}]\Big\|_1\\
                                   &+\Big\|B^{z_2-1}\cdot [BA^{\frac{1}{2}},A^{z_1-\frac{1}{2}+is}-A^{z_2-\frac{1}{2}+is}]\Big\|_1.
\end{align*}
By the H\"older-type inequality \eqref{weird_holder}, we have
\begin{align*}
    \Big\|T_{z_1}(s)-T_{z_2}(s)\Big\|_1&\leq \Big\|[BA^{\frac{1}{2}},A^{\frac{1}{2}+is}]\Big\|_{\frac{p}{2},\infty}\Big\|Y^{z_1-1}-Y^{z_2-1}\Big\|_{\frac{p}{p-2},1}\\
                                       &+\Big\|B^{z_1-1}-B^{z_2-1}\Big\|_{\frac{p}{p-2},1}\Big\|[BA^{\frac{1}{2}},A^{z_1-\frac{1}{2}+is}]\Big\|_{\frac{p}{2},\infty}\\
                                       &+\Big\|B^{z_2-1}\Big\|_{\frac{p}{p-2},1}\Big\|[BA^{\frac{1}{2}},A^{z_1-\frac{1}{2}+is}-A^{z_2-\frac{1}{2}+is}]\Big\|_{\frac{p}{2},\infty}.
\end{align*}

By Lemma \ref{first psacta lemma}, there exists a constant $c_p'>0$ such that
$$\|[BA^{\frac{1}{2}},A^{\frac{1}{2}+is}]\|_{\frac{p}{2},\infty}\leq c_p'(1+|s|)\|[BA^{\frac{1}{2}},A^{\frac12}]\|_{\frac{p}{2},\infty},$$
$$\|[BA^{\frac{1}{2}},A^{z_1-\frac{1}{2}+is}]\|_{\frac{p}{2},\infty}\leq c_p'(1+|z_1|+|s|)\|[BA^{\frac{1}{2}},A^{\frac12}]\|_{\frac{p}{2},\infty}.$$
Also, by normalisation,
$$\Big\|B^{z_2-1}\Big\|_{\frac{p}{p-2},1}\leq \Big\|B^{z_2-p}\|_{\infty}\Big\|B^{p-1}\Big\|_{\frac{p}{p-2},1}\stackrel{\leq} c_p'\|B\|_{p,\infty}^{p-1}=:c_p''.$$
Therefore,
\begin{align*}
    \Big\|T_{z_1}(s)-T_{z_2}(s)\Big\|_1&\leq c_p'(1+|s|)\Big\|[BA^{\frac{1}{2}},A^{\frac{1}{2}}]\Big\|_{\frac{p}{2},\infty}\Big\|Y^{z_1-1}-Y^{z_2-1}\Big\|_{\frac{p}{p-2},1}\\
    &+c_p'(1+|z_1|+|s|)\Big\|B^{z_1-1}-B^{z_2-1}\Big\|_{\frac{p}{p-2},1}\Big\|[BA^{\frac{1}{2}},A^{\frac{1}{2}}]\Big\|_{\frac{p}{2},\infty}\\
    &+c_p''\Big\|[BA^{\frac{1}{2}},A^{z_1-\frac{1}{2}+is}-A^{z_2-\frac{1}{2}+is}]\Big\|_{\frac{p}{2},\infty}.
\end{align*}

By Lemma \ref{second psacta lemma}, there is some $c_p'''$ such that
$$\Big\|[BA^{\frac{1}{2}},A^{z_1-\frac{1}{2}+is}-A^{z_2-\frac{1}{2}+is}]\Big\|_{\frac{p}{2},\infty}\leq c_p'''|z_1-z_2|\cdot(1+|z_1|+|z_2|+|s|)\Big\|[BA^{\frac{1}{2}},A^{\frac{1}{2}}]\Big\|_{\frac{p}{2},\infty},$$
$$\Big\|Y^{z_1-1}-Y^{z_2-1}\Big\|_{\frac{p}{p-2},1},\Big\|B^{z_1-1}-B^{z_2-1}\Big\|_{\frac{p}{p-2},1}\leq c_p'''|z_1-z_2|.$$
Therefore,
\begin{align*}
    \Big\|T_{z_1}(s)-T_{z_2}(s)\Big\|_1&\leq c_p'c_p'''(1+|s|)|z_1-z_2|\Big\|[BA^{\frac{1}{2}},A^{\frac{1}{2}}]\Big\|_{\frac{p}{2},\infty}\\
    &+c_p'c_p'''(1+|z_1|+|s|)|z_1-z_2|\Big\|[BA^{\frac{1}{2}},A^{\frac{1}{2}}]\Big\|_{\frac{p}{2},\infty}\\
    &+c_p''c_p'''|z_1-z_2|(1+|z_1|+|z_2|+|s|)\Big\|[BA^{\frac{1}{2}},A^{\frac{1}{2}}]\Big\|_{\frac{p}{2},\infty}.
\end{align*}
This completes the proof.
\end{proof}

In addition to controlling the $\Lc_1$ norm of $T_{z_1}(s)-T_{z_2}(s)$, we also need good control on the function $g_z$ appearing in Theorem \ref{csz key lemma},
which the next lemma provides. Denote by $W^2_2$ the Sobolev space of functions on $\Rl$ such that
\[
    \|g\|_{W^{2}_2}^2 = \|g\|_2^2+\|g'\|_2^2+\|g''\|_2^2 < \infty.
\]

\begin{lemma}\label{gz continuity lemma} There exists $n\in\mathbb{N}$ such that
$$\|g_z\|_{W^2_2}\leq c_{{\rm abs}}(1+|z|)^{n-1},\quad\re(z)\geq 2,$$
$$\|g_{z_1}-g_{z_2}\|_{W^2_2}\leq c_{{\rm abs}}|z_1-z_2|\cdot (1+|z_1|)^{n-1}\cdot (1+|z_2|)^{n-1},\quad\re(z_1),\re(z_2)\geq 2.$$
\end{lemma}
\begin{proof}
       For $s\neq 0$ and $\re(z)>1,$ we rewrite the explicit form of $g_z$ in Theorem \ref{csz key lemma} as follows: 
    \[
        g_z(s) = \frac{1}{2}\left(\frac{\tanh(\frac{z-1}{2}s)}{\tanh(\frac{s}{2})}-1\right).
    \]
It is immediate that
$$\|g_z\|_{W^2_2}^2=\|g_z\|_{W^2_2([-1,1])}^2+\|g_z\|_{W^2_2([1,\infty))}^2+\|g_z\|_{W^2_2((-\infty,1])}^2,$$
$$\|g_{z_1}-g_{z_2}\|_{W^2_2}^2=\|g_{z_1}-g_{z_2}\|_{W^2_2([-1,1])}^2+\|g_{z_1}-g_{z_2}\|_{W^2_2([1,\infty))}^2+\|g_{z_1}-g_{z_2}\|_{W^2_2((-\infty,-1])}^2.$$

On $[-1,1],$ we write
$$g_z=\frac12\Big(f_z\cdot f_2^{-1}-1\Big),$$
where
$$f_z(s)=\frac{\tanh(\frac{z-1}{2}s)}{s},\quad s\in[-1,1].$$
By the Leibniz rule, it suffices to verify 
$$\|f_z\|_{W^2_2([-1,1])}\leq c_{{\rm abs}}(1+|z|)^{n-1},\quad\re(z)\geq 2,$$
$$\|f_{z_1}-f_{z_2}\|_{W^2_2([-1,1])}\leq c_{{\rm abs}}|z_1-z_2|\cdot (1+|z_1|)^{n-1}\cdot (1+|z_2|)^{n-1},\quad\re(z_1),\re(z_2)\geq 2.$$
This computation is omitted.

On $[1,\infty),$ we write
$$g_z=(h_z-h_2)\cdot\frac1{2(1+h_2)},$$
where
$$h_z(s)=\tanh(\frac{z-1}{2}s)-1,\quad s\in[1,\infty).$$
By the Leibniz rule, it suffices to verify 
$$\|h_z\|_{W^2_2([1,\infty))}\leq c_{{\rm abs}}(1+|z|)^{n-1},\quad\re(z)\geq 2,$$
$$\|h_{z_1}-h_{z_2}\|_{W^2_2([1,\infty))}\leq c_{{\rm abs}}|z_1-z_2|\cdot (1+|z_1|)^{n-1}\cdot (1+|z_2|)^{n-1},\quad\re(z_1),\re(z_2)\geq 2.$$
This computation is omitted.

On $(-\infty,-1],$ the argument is exactly the same as on $[1,\infty).$
\end{proof}

\subsection{Proof of Theorem \ref{main_analytic_theorem}}
We now can complete the proof of Theorem \ref{main_analytic_theorem}. We will make use of the inequality
\begin{equation}\label{sobolev_ineq}
    \|\widehat{f}\|_1,\|\widehat{f'}\|_1 \leq c_{\mathrm{abs}}\|f\|_{W^2_2},\quad f \in W^2_2(\mathbb{R}).
\end{equation}
Indeed, denoting $H(s) = (1+s^2)^{\frac12},$ the Cauchy-Schwarz inequality and the Plancherel theorem imply that
\[
    \|\widehat{f'}\|_1 = \|H\frac{1}{H}\widehat{f'}\|_1 \leq \|H^{-1}\|_{2}\|H\widehat{f'}\|_2\leq c_{\mathrm{abs}}\|f\|_{W^2_2}.
\]
Similarly, $\|\widehat{f}\|_1 \leq c_{\mathrm{abs}}\|f\|_{W^1_2} \leq c_{\mathrm{abs}}\|f\|_{W^2_2}.$

\begin{proof}[Proof of Theorem \ref{main_analytic_theorem}] 
By Theorem \ref{csz key lemma}, the mapping
$$s\mapsto T_z(s)\widehat{g}_z(s)\quad \re(z)>p$$
is continuous in the weak operator topology. By Lemma \ref{tz boundedness lemma}, we have
$$\|T_z(s)\widehat{g}_z(s)\|_1\leq c_{p,A,B}(1+|z|)(1+|s|)\cdot|\widehat{g}_z(s)|.$$
Since $g_z$ belongs to the Schwartz class, we have
\[
    \int_{-\infty}^\infty \|T_z(s)\widehat{g}_z(s)\|_1 \,ds < \infty.
\]

It follows (see e.g. \cite[Lemma 2.3.2]{SZ-asterisque}) that the integral of the function $s\mapsto T_z(s)\widehat{g}_z(s)$  exists
in the weak operator topology and
$$\int_{\mathbb{R}}T_z(s)\widehat{g}_z(s)ds\in\mathcal{L}_1.$$
The same result implies that
$${\rm Tr}\Big(\int_{\mathbb{R}}T_z(s)\widehat{g}_z(s)ds\Big)=\int_{\mathbb{R}}{\rm Tr}(T_z(s))\widehat{g}_z(s)ds.$$
From Lemma \ref{tz boundedness lemma}, $T_z(0)\in \Lc_1$ and hence
$${\rm Tr}\Big(T_z(0)-\int_{\mathbb{R}}T_z(s)\widehat{g}_z(s)ds\Big)={\rm Tr}(T_z(0))-\int_{\mathbb{R}}{\rm Tr}(T_z(s))\widehat{g}_z(s)ds.$$
By Theorem \ref{csz key lemma} and the definition of $f_{A,B},$ we have
$$f_{A,B}(z)={\rm Tr}\Big(T_z(0)-\int_{\mathbb{R}}T_z(s)\widehat{g}_z(s)ds\Big),\quad\re(z)>p.$$
Thus,
\begin{equation}\label{f extended def}
f_{A,B}(z)={\rm Tr}(T_z(0))-\int_{\mathbb{R}}{\rm Tr}(T_z(s))\widehat{g}_z(s)ds,\quad\re(z)>p.
\end{equation}
Let us show that $f_{A,B}$ satisfies the required inequality. Indeed, we have
\begin{align*}
|f_{A,B}(z_1)-f_{A,B}(z_2)|&\leq|{\rm Tr}(T_{z_1}(0))-{\rm Tr}(T_{z_2}(0))|+\int_{\mathbb{R}}|{\rm Tr}(T_{z_1}(s))-{\rm Tr}(T_{z_2}(s))|\cdot |\widehat{g}_{z_1}(s)|ds\\
               &\quad+\int_{\mathbb{R}}|{\rm Tr}(T_{z_2}(s))|\cdot |\widehat{g}_{z_1}(s)-\widehat{g}_{z_1}(s)|ds\\
               &\leq\|T_{z_1}(0)-T_{z_2}(0)\|_1+\int_{\mathbb{R}}\|T_{z_1}(s)-T_{z_2}(s)\|_1\cdot |\widehat{g}_{z_1}(s)|ds\\
               &\quad +\int_{\mathbb{R}}\|T_{z_2}(s)\|_1\cdot |\widehat{g}_{z_1}(s)-\widehat{g}_{z_1}(s)|ds.
\end{align*}
Using Lemmas \ref{tz boundedness lemma} and \ref{tz continuity lemma}, we obtain
\begin{align*}
|f_{A,B}(z_1)-f_{A,B}(z_2)| &\leq c_{p,A,B}(1+|z_1|+|z_2|)\cdot |z_1-z_2|\\
                &\quad +c_{p,A,B}|z_1-z_2|(1+|z_1|+|z_2|)\int_{\mathbb{R}}(1+|s|)\cdot |\widehat{g}_{z_1}(s)|ds\\
                &\quad +c_{p,A,B}(1+|z_2|)\cdot\int_{\mathbb{R}}(1+|s|)\cdot |\widehat{g}_{z_1}(s)-\widehat{g}_{z_1}(s)|ds.
\end{align*}
Now, from \eqref{sobolev_ineq}
\begin{align*} 
\int_{\mathbb{R}}(1+|s|)\cdot |\widehat{g}_{z_1}(s)|ds &= \|\widehat{g}_{z_1}\|_1+\|\widehat{g'_{z_1}}\|_1\leq c_{{\rm abs}}\|g_z\|_{W^2_2},\\
\int_{\mathbb{R}}(1+|s|)\cdot |\widehat{g}_{z_1}(s)-\widehat{g}_{z_1}(s)|ds &=\|\widehat{g}_{z_1}-\widehat{g}_{z_2}\|_1+\|\widehat{g'_{z_1}}-\widehat{g'_{z_2}}\|_1\leq c_{{\rm abs}}\|g_{z_1}-g_{z_2}\|_{W^2_2}.
\end{align*}
The desired bound on $|f_{A,B}(z_1)-f_{A,B}(z_2)|$ follows now from Lemma \ref{gz continuity lemma}.
\end{proof}

\section{Applications to Connes' integration formula}
    We now explain in several examples how the preceding Tauberian theorems imply exact asymptotic
    formulae for the spectra of certain operators on noncommutative spaces.
    
    The following definition is standard, and follows e.g. \cite{ConnesSpectral}, \cite[Definition 2]{CPRS1}, \cite[Chapter 10]{green-book}.
    \begin{definition}\label{spectral_triple_definition}
        A spectral triple $(\Ac,H,D)$ consists of the following data:
        \begin{enumerate}[{\rm (i)}]
            \item{} a separable Hilbert space $H,$
            \item{} a $*$-algebra $\Ac$ of bounded linear endomorphisms of $H$, containing the identity operator,
            \item{} and a self-adjoint operator $D$ on $H$, such that for all $a \in \Ac,$ $a(\mathrm{dom}(D))\subset \mathrm{dom}(D)$, and the commutator
            \[
                [D,a] = Da-aD:\mathrm{dom}(D)\to H
            \]
            has bounded extension.
        \end{enumerate}
        A spectral triple $(\Ac,H,D)$ is said to be $\Lc_{p,\infty}$-summable, for $p>0$, if 
        \[
            (i+D)^{-1} \in \Lc_{p,\infty}(H).
        \]
        Equivalently, $(1+D^2)^{-\frac{1}{2}}\in \Lc_{p,\infty}(H)$ or $(i-D)^{-1}\in \Lc_{p,\infty}(H).$
        A spectral triple $(\Ac,H,D)$ is said to be $QC^1$ if the operators
        \[
            \delta(a) := [(1+D^2)^{\frac{1}{2}},a]:\mathrm{dom}(D)\to H,\quad a \in \Ac.
        \]
        also have bounded extension. Equivalently, $[|D|,a]$ has bounded extension for every $a \in \Ac.$
    \end{definition}
    It is common to make further assumptions on the spectral triple, especially relating to the meromorphic continuation of the $\zeta$-functions
    \[
        z\mapsto \Tr(b(1+D^2)^{-\frac{z}{2}}),\quad \re(z)>p,\; b=a,\delta(a),\quad a \in \Ac.
    \]
    We will work with spectral triples obeying the following condition:
    \begin{condition}\label{zeta_condition}
        For all $0\leq a \in \Ac$, there exists $c \in \Cplx$ such that the function
        \[
            z\mapsto \Tr(a^z(1+D^2)^{-\frac{z}{2}})-\frac{c}{z-p},\quad \re(z)>p
        \]
        admits a continuous extension to the closed half-plane $\re(z)\geq p.$
    \end{condition}
    This should be compared with the condition of discrete dimension spectrum of Connes and Moscovici \cite[Definition II.1]{CM1995} and related notions such as isolated dimension spectrum \cite[Definition 3.1]{CGRS2}.
    Condition \ref{zeta_condition} is not obviously implied by either of these assumptions, however the condition is at least easy to verify in a number of cases, as we demonstrate below.
    
    \begin{example}
        The main commutative example is algebra $\Ac = C^\infty(X)$, where $X$ is a smooth closed Riemannian spin manifold $X$, Hilbert space $H=L_2(X,E)$ of square-integrable
        sections of a spinor bundle $E$ on $X$ and Dirac-type operator $D$ acting on section of $E.$ Here, $\Ac$ is considered to act
        on $H$ by pointwise multiplication, $M_f\xi = f\xi,$ where $f \in \Ac$ and $\xi\in H.$ In this case, $(\Ac,H,D)$ is $\Lc_{d,\infty}$-summable, where $d$ is the dimension of the manifold $X.$
        For discussion of this class of examples, see e.g. \cite[Section 11.1]{green-book}.
                
        Standard results in elliptic theory \cite[Theorem 12.1]{Shubin-psido-2001}, \cite[Section 4.4]{Grubb1996} imply that for $0\leq f \in C^\infty(X)$
        \[
            z\mapsto \Tr(M_f^z(1+D^2)^{-\frac{z}{2}}),\quad \re(z)>d
        \]
        admits a meromorphic continuation to $\Cplx$, with only simple poles located at some subset of the points
        \[
            \{d,d-2,d-4,\ldots\}.
        \]
        For $f\geq 0,$ denote $c = \res_{z=d} \Tr(M_f^z(1+D^2)^{-\frac{z}{2}})$. It follows that the function
        \[
            \Tr(M_f^z(1+D^2)^{-\frac{z}{2}})-\frac{c}{z-d},\quad \re(z)>d
        \]
        admits a continuous extension to the half-plane $\re(z)\geq d.$
        Hence, in this case, Condition \ref{zeta_condition} is satisfied. 
    \end{example}
    
    Since the algebra $\Ac$ of a spectral triple $(\Ac,H,D)$ is in general not closed under continuous functional calculus,
    it is helpful to complete $\Ac$ in $\Lc_{\infty}(H).$ 
    \begin{definition}
        Let $(\Ac,H,D)$ be a spectral triple. We denote by $\overline{\Ac}$ the $C^*$-subalgebra of $\Lc_{\infty}(H)$
        generated by $\Ac.$
    \end{definition}
    
    \begin{remark}\label{positive_approximation_remark}
        By definition, every $a \in \overline{\Ac}$ is the limit in operator norm of a sequence $\{a_n\}_{n=0}^\infty$ in $\Ac.$
        We note that if $a$ is positive, then it is possible to select every $a_n$ positive.
        Indeed, since $a$ is positive, there exists
        $b \in \overline{\Ac}$ such that $a = b^*b$, and since $b\in \overline{\Ac}$ there exists a sequence $\{b_n\}_{n\geq 0}$ in $\Ac$ such that $b_n\to b.$ Defining $a_n = b_n^*b_n$
        gives
        \[
            \|a-a_n\| = \|b^*b-b_n^*b_n\| \leq \|b^*(b-b_n)\|+\|b_n^*(b-b_n)\|.
        \]
        Since $\lim_{n\to\infty} \|b_n\|=\|b\|,$ it follows that
        \[
            \limsup_{n\to\infty} \|a-a_n\| \leq \|b\| \limsup_{n\to\infty} \|b-b_n\| = 0.
        \]
    \end{remark}

    We now recall the statement of Theorem \ref{main_spectral_triple_theorem} and give a proof.
    \begin{proposition}\label{abstract_cif}
        Let $(\Ac,H,D)$ be a spectral triple satisfying Condition \ref{zeta_condition},
        For all $0\leq a \in \overline{\Ac},$ there exists the limit
        \[
            \lim_{t\to\infty} t\mu(t,(1+D^2)^{-\frac{p}{4}}a(1+D^2)^{-\frac{p}{4}}).
        \]
        For $a=x^{2p},$ where $0\leq x \in \Ac,$ this limit is given by
        \[
            \lim_{t\to\infty} t\mu(t,(1+D^2)^{-\frac{p}{4}}x^{2p}(1+D^2)^{-\frac{p}{4}}) = \frac{1}{p}\lim_{z\downarrow p}(z-p)\Tr(x^{2z}(1+D^2)^{-\frac{z}{2}}.
        \]
    \end{proposition}
    \begin{proof}
        Initially we assume that $a = x^{2p},$ where $0 \leq x \in \Ac.$ For this case, we will apply Theorem \ref{tauberian_corollary} to
        $A = x^2$ and $B = (1+D^2)^{-\frac12}.$ Part \eqref{ccond1} of Condition \ref{compact conditions for analyticity} follows
        from the assumption that $(\Ac,H,D)$ is $\Lc_{p,\infty}$-summable. By the definition of $QC^1,$ $[(1+D^2)^{\frac12},x]$ is bounded.
        Hence,
        \[
            [B,A^{\frac12}] = [(1+D^2)^{-\frac12},x] = -(1+D^2)^{-\frac12}[(1+D^2)^{\frac12},x](1+D^2)^{-\frac12}.
        \]
        This belongs to $\Lc_{\frac{p}{2},\infty},$ by the H\"older inequality. This verifies part \eqref{ccond2} of Condition \ref{compact conditions for analyticity}.
        
        Hence, by Theorem \ref{tauberian_corollary} there exists the limit
        \[
            \lim_{t\to\infty} t\mu(t,B^{\frac{p}{2}}A^pB^{\frac{p}{2}}) = \lim_{t\to\infty} t\mu(t,(1+D^2)^{-\frac{p}{4}}x^{2p}(1+D^2)^{-\frac{p}{4}}) = \frac{1}{p}\lim_{z\downarrow p}(z-p)\Tr(x^{2z}(1+D^2)^{-\frac{z}{2}}.
        \]        
        We now prove that the limit on the left hand side above still exists for all $0\leq x \in \overline{\Ac}.$ Indeed, from Remark \ref{positive_approximation_remark}, there exists a sequence $\{x_k\}_{k\geq 0}$
        in $\Ac$ such that every $x_k$ is positive and $x_k\to x.$ It follows that $x_k^{2p}\to x^{2p}$ in the operator norm, and hence that
        \[
            (1+D^2)^{-\frac{p}{4}}x_k^{2p}(1+D^2)^{-\frac{p}{4}}\to (1+D^2)^{-\frac{p}{4}}x^{2p}(1+D^2)^{-\frac{p}{4}}
        \]
        in $\Lc_{1,\infty}.$ It now follows from Lemma \ref{bs_perturbation_lemma_sequential} with $X_k = (1+D^2)^{-\frac{p}{4}}x_k^{2p}(1+D^2)^{-\frac{p}{4}}$ that there exists the limit
        \[
            \lim_{t\to\infty} t\mu(t,(1+D^2)^{-\frac{p}{4}}x^{2p}(1+D^2)^{-\frac{p}{4}}).
        \]
        Taking $x = a^{\frac{1}{2p}} \in \overline{\Ac}$ completes the proof.
    
    \end{proof}
%
%
    
    In order to prove Theorem \ref{secondary_spectral_triple_theorem}, we will use the following lemma, which is proved in \cite[Lemma 5.1]{SZ-asymptotics}.
    \begin{lemma}\label{SZ_lemma}
        If $B\in \Lc_{2,\infty}$ and $A \in \Lc_{\infty}$ are self-adjoint and such that $[A,B] \in (\Lc_{2,\infty})_0,$ then
        \[
            (BAB)_+-BA_+B \in (\Lc_{1,\infty})_0.
        \]
    \end{lemma}
    
    We will apply this with $A = a\in \overline{\Ac}$ and $B = (1+D^2)^{-\frac{p}{4}}.$ To verify the assumptions of the lemma, we require the following.    
    \begin{lemma}\label{section 5 cts lemma}
        Let $(\Ac,H,D)$ be an $\Lc_{p,\infty}$-summable $QC^1$ spectral triple. For all $a \in \overline{\Ac}$ we have
        \[
            [(1+D^2)^{-\frac{1}{2}},a]\in (\Lc_{p,\infty})_0.
        \]
    \end{lemma}
    \begin{proof} 
        Since $(\Ac,H,D)$ is $QC^1$, for all $a \in \Ac$ the operator $\delta(a) = [(1+D^2)^{\frac12},a]$ is bounded on $H.$ Thus,
        when $a \in \Ac$ we have
        \[
            [(1+D^2)^{-\frac{1}{2}},a] = -(1+D^2)^{-\frac{1}{2}}\delta(a)(1+D^2)^{-\frac{1}{2}} \in \Lc_{p,\infty}\cdot\Lc_{\infty}\cdot\Lc_{p,\infty}.
        \]
        This is contained in $\Lc_{\frac{p}{2},\infty}\subset (\Lc_{p,\infty})_0$, by H\"older's inequality. This proves the result for $a\in \Ac.$ By the quasi-triangle inequality
        for $\Lc_{p,\infty},$ we have
        \[
            \|[(1+D^2)^{-\frac{1}{2}},a]\|_{p,\infty} \leq c_p\|(1+D^2)^{-\frac{1}{2}}\|_{p,\infty}\|a\|.
        \]
        Hence the linear mapping
        \[
            \Ac \ni a \mapsto [(1+D^2)^{-\frac{1}{2}},a] \in (\Lc_{p,\infty})_0
        \]
        is continuous from the operator norm on $\Ac$ to $\Lc_{p,\infty}.$ Thus, 
        for $a \in \overline{\Ac}$ the commutator $[(1+D^2)^{-\frac{1}{2}},a]$ belongs
        to the closure of $\Lc_{\frac{p}{2},\infty}$ in the $\Lc_{p,\infty}$-quasinorm, which is $(\Lc_{p,\infty})_0.$
    \end{proof}
    
%
%
    
    \begin{proof}[Proof of Theorem \ref{secondary_spectral_triple_theorem}]
        We prove the result involving the positive part, the argument for the negative parts is identical. Let $a \in \overline{A}.$ Since the function $t\mapsto t_+$ is continuous, we have $a_+\in \overline{\Ac}.$ Applying Theorem \ref{main_spectral_triple_theorem} gives the existence of the limit
        \[
            \lim_{t\to\infty} t\mu(t,(1+D^2)^{-\frac{p}{4}}a_+(1+D^2)^{-\frac{p}{4}}).
        \]  
        By Lemma \ref{section 5 cts lemma}, we have
        \[
            [(1+D^2)^{-\frac{1}{2}},a] \in (\Lc_{p,\infty})_0.
        \]
        Applying Lemma \ref{HSZ_modified_inequality}, it follows that
        \[
            [(1+D^2)^{-\frac{p}{4}},a] \in (\Lc_{2,\infty})_0.
        \]
        This verifies the assumption of Lemma \ref{SZ_lemma} with $B = (1+D^2)^{-\frac{p}{4}}$ and $A = a.$
        Lemma \ref{SZ_lemma} for this case implies
        \[
            \left((1+D^2)^{-\frac{p}{4}}a(1+D^2)^{-\frac{p}{4}}\right)_+-(1+D^2)^{-\frac{p}{4}}a_+(1+D^2)^{-\frac{p}{4}} \in (\Lc_{1,\infty})_0.
        \]
        Hence, Corollary \ref{simplified_bs_perturbation_lemma} yields 
        \[
            \lim_{t\to\infty} t\mu(t,\left((1+D^2)^{-\frac{p}{4}}a(1+D^2)^{-\frac{p}{4}}\right)_+) = \lim_{t\to\infty} t\mu(t,(1+D^2)^{-\frac{p}{4}}a_+(1+D^2)^{-\frac{p}{4}}).
        \]
    \end{proof}

\subsection{Noncommutative tori}
We denote $\Circ$ for the unit circle, thought of as the quotient space
\[
    \Circ = \Rl/(2\pi \Itgr).
\]
Let $d\geq 2$ and let $\theta$ be an antisymmetric real $d\times d$ matrix. The \emph{noncommutative torus} $\Circ^d_\theta$ (otherwise known as a quantum torus)
is a heavily studied noncommutative space in the sense of Connes. The $C^*$-algebra $C(\Circ^d_\theta)$ may be described as the universal $C^*$-algebra generated by $d$ unitary generators
$\{U_j\}_{j=1}^d$ obeying the relations
\[
    U_jU_k = e^{2\pi i \theta_{j,k}}U_kU_j,\quad 1\leq j\leq d.
\]
Noncommutative tori were considered by Rieffel as an example in the theory of deformation quantization \cite{Rieffel-1981} and even earlier their presence can be seen in the work of Effros and Hahn in the 1960s \cite{Effros-Hahn-memoirs-1967}. These algebras later appeared prominently as an example of a $C^*$-algebra associated to a foliation by Connes \cite[Chapter 2, Section 9.$\beta$]{Connes1994}. 

Harmonic analysis on $\Circ^d_\theta$ has been investigated by many authors. In particular, pseudodifferential operator theory has been developed in some detail, beginning with the early
work of Baaj \cite{Baaj1988} and Connes \cite{Connes1980}, and has now reached a state of maturity, see \cite{HLP2019a,HLP2019b,LeePonge2020,HaPonge2020,PongeCTT2020,PongeRes2020,XX2018,Tao2018,LJP2016,MSX2019,MSZ2018}. For the theory of function spaces on $\Circ^d_\theta$, we refer the reader to \cite{XXY2018,Spera1992,Tao2018}. The noncommutative geometry of noncommutative tori was developed by Connes \cite{Connes1980}, see also the exposition \cite[Chapter 12]{green-book}.
There have been some references to these spaces in mathematical physics \cite{Bellissard-original,Bellissard-van-Elst-Schulz-Baldes,Connes-Douglas-Schwarz-matrix-theory-1998}, and more recently there has been significant interest in geometric aspects of quantum tori \cite{CM2014,CT2011}.

Equivalently, we may define $C(\Circ^d_\theta)$ in terms of its representation
on a specific Hilbert space. 
We define $C(\Circ^d_\theta)$ as the $C^*$-subalgebra,
\[
    C(\Circ^d_\theta)\subset \Lc_{\infty}(L_2(\Circ^d))
\]
generated by the unitary operators
\begin{equation}\label{nc_torus_generator_def}
    U_j\xi(t) = e^{it_j}\xi(t+\pi \theta e_j),\quad j=1,\ldots,d,\; t \in \mathbb{T}^d.
\end{equation}
Here, $e_j$ is the $j$th standard basis vector of $\Rl^d$ and $\theta e_j$ is the $j$th column of the matrix $\theta.$ Observe that when $\theta=0$ this reduces to the description
of $C(\Circ^d)$ as being the $C^*$-subalgebra of pointwise multipliers of $L_2(\Circ^d)$ generated by the trigonometric basis functions $t\mapsto \exp(it_j).$

A tracial state on $C(\Circ^d_\theta)$ may be defined in terms of its representation
on $L_2(\Circ^d)$ by
\[
    \tau(x) := \frac{1}{(2\pi)^d}\int_{\Circ^d} (x1)(t)\,dt,\quad x \in C(\Circ^d_\theta)
\]
Here, $x1$ is the action of $x \in C(\Circ^d_\theta)$ on the constant function $1\in L_2(\Circ^d).$ The GNS representation space for $C(\Circ^d_\theta)$ is denoted $L_2(\Circ^d_\theta)$,
and the weak closure (or double commutant) of $C(\Circ^d_\theta)$ in its GNS representation is a finite von Neumann algebra $L_{\infty}(\Circ^d_\theta).$

We adopt the shorthand notation
\[
    U^n := U_1^{n_1}U_2^{n_2}\cdots U_d^{n_d},\quad n = (n_1,\ldots,n_d)\in \mathbb{Z}^d.
\]
The set $\{U^n\}_{n\in \Itgr^d}$ is called the trigonometric basis for $L_2(\Circ^d_\theta)$, and forms a complete orthonormal system. 
The coefficients $\{\widehat{x}(n)\}_{n\in \Itgr^d}$ of $x\in L_2(\Circ^d_\theta)$ in the expansion
\[
    x = \sum_{n\in \Itgr^d} \widehat{x}(n)U^n
\]
are called the Fourier coefficients of $x.$

A strongly continuous action $\alpha$ the group $\Circ^d$ on $L_{\infty}(\Circ^d_\theta)$ by isometries may be defined such that
\[
    \alpha_t(U_j) = e^{it_j}U_j,\quad j=1,\ldots,d.
\]
An element $x \in C(\Circ^d_\theta)$ is called \emph{smooth} if the function
\[
    t\mapsto \alpha_t(x)
\]
is a smooth $C(\Circ^d_\theta)$-valued function. Equivalently, the sequence $\{\widehat{x}(n)\}_{n\in \Itgr^d}$ of Fourier coefficients of $x$
has rapid decay. Denote the set of smooth elements as $C^{\infty}(\Circ^d_\theta).$

The partial derivations $\partial_j,$ $j=1,\ldots,d$ are defined on the trigonometric basis by
\[
    \partial_j(U^n) = n_jU^n,\quad n\in \Itgr^d.
\]
For $\alpha \in \Ntrl^d,$ we denote $\partial^{\alpha}$ for $\partial_1^{\alpha_1}\cdots\partial_d^{\alpha_d}.$ The Laplace operator $\Delta$ is defined as
\[
    \Delta = -\sum_{j=1}^d \partial_j^2.
\]
On a trigonometric basis element, $\Delta$ acts as $\Delta U^n = -|n|^2U^n$, where $|n| = (n_1^2+\cdots+n_d^2)^{\frac{1}{2}}.$ 

For $s \in \Rl$, the operator $(1-\Delta)^{\frac{s}{2}}$ may be defined on the trigonometric basis
\[
    (1-\Delta)^{\frac{s}{2}}U^n = (1+|n|^2)^{\frac{s}{2}}U^n.
\]
For $s \leq 0$ this is a bounded linear operator on $L_2(\Circ^d_\theta)$ and the analogy of the Bessel potential for $\Circ^d_\theta.$

Given $x \in L_{\infty}(\Circ^d_\theta)$, we denote by $\rho(x)$ the action of $x$ on $L_2(\Circ^d_\theta)$ by left-multiplication.
That is,
\[
    \rho(x)\xi = x\xi,\quad x \in L_{\infty}(\Circ^d_\theta),\;\xi \in L_2(\Circ^d_\theta).
\]

In order to apply Theorems \ref{main_spectral_triple_theorem} and \ref{scwl_spectral_triple_thm} to $\Circ^d_\theta,$ we use the following facts.
\begin{lemma}\label{facts_about_nctori}
    Let $x\in C^\infty(\Circ^d_\theta),$ and $-\infty < \alpha < 1.$
    \begin{enumerate}[{\rm (a)}]
        \item{} For all $\beta > 0$, we have $(1-\Delta)^{-\frac{\beta}{2}} \in \Lc_{\frac{d}{\beta},\infty}.$
        \item{} For all $-\infty<\alpha<1$, we have $[\rho(x),(1-\Delta)^{\frac{\alpha}{2}}] \in \Lc_{\frac{d}{1-\alpha},\infty}.$
    \end{enumerate}
\end{lemma}
The first assertion of Lemma \ref{facts_about_nctori} is a straightforward consequence of the explicit spectral decomposition $(1-\Delta)^{-\frac{\beta}{2}}U^n = (1+|n|^2)^{-\frac{\beta}{2}}U^n$.
The second assertion requires an argument, which was given in Section 5 of \cite{MSX2019}, we refer the reader to \cite[Corollary 5.5]{MSX2019} there. 

A standard example of a spectral triple for $\Circ^d_\theta$ is constructed as an isospectral deformation of the spin Dirac operator on the torus. Let $N = \lfloor \frac{d}{2}\rfloor,$ and
$H = L_2(\Circ^d_\theta)\otimes \Cplx^N,$ and let $\Ac$ be the subalgebra of $\Lc_\infty(H)$ of left multiplication by $C^\infty(\Circ^d_\theta)$ on the first tensor factor. Select a family $\{\gamma_j\}_{j=1}^d$ of $N\times N$ self-adjoint matrices such that $\gamma_j\gamma_k+\gamma_k\gamma_j=2\delta_{j,k}$ for $1\leq j,k\leq d.$ The standard spin Dirac operator for $\Circ^d_\theta$ is the linear operator
\[
    D = \sum_{j=1}^d \partial_j\otimes\gamma_j.
\]
It is well-known that $(\Ac,H,D)$ is a $QC^1$-spectral triple \cite[Chapter 12]{green-book}.

\begin{lemma}\label{zeta_function_lemma}
    The function
    \[
        F(z) := \Tr((1-\Delta)^{-\frac{z}{2}}),\quad \re(z) > d
    \]
    admits meromorphic continuation with a simple pole at $z=d$ and
    \[
        \res_{z=d} F(z) = \Vol(S^{d-1}).
    \]
\end{lemma}
\begin{proof}
    The trace $\Tr((1-\Delta)^{-\frac{z}{2}})$ can be expanded in terms of the trigonometric basis as
    \[
        F(z) = \sum_{n\in \Itgr^d} (1+|n|^2)^{-\frac{z}{2}},\quad \re(z)>d.
    \]
    This is precisely the same as the trace of the operator $(1-\Delta)^{-\frac{z}{2}}$ on the ordinary (commutative) $d$-torus.
    The meromorphic continuation in this case is well-known. For a much more general result see e.g. \cite[Chapter 2]{Shubin-psido-2001}.
\end{proof}

\begin{lemma}\label{torus_analyticity_check}
    Let $0\leq a \in C^{\infty}(\Circ^d_\theta),$. The function
    \[
        z\mapsto \Tr(\rho(a)^z(1-\Delta)^{-\frac{z}{2}}),\quad \re(z)>d
    \]
    admits a meromorphic continuation to the half-plane $\re(z)>0$, with only a simple pole at $z=d$
    and corresponding residue
    \[
        \res_{z=d} \Tr(\rho(a)^z(1-\Delta)^{-\frac{z}{2}}) = \Vol(S^{d-1})\tau(a^d).
    \]
\end{lemma}
\begin{proof}
    Denote $F$ for the function
    \[
        F(z) = \Tr(\rho(a)^z(1-\Delta)^{-\frac{z}{2}}).
    \]
    Expressing the trace in the trigonometric basis $\{U^n\}_{n\in \mathbb{Z}^d}$ we have
    \begin{align*}
        F(z) &= \sum_{n\in\mathbb{Z}^d} \langle \rho(a)^z(1-\Delta)^{-\frac{z}{2}}U^n,U^n\rangle\\
             &= \sum_{n\in \mathbb{Z}^d} \tau((U^n)^*U^na^z)(1+|n|^2)^{-\frac{z}{2}}\\
             &= \sum_{n\in \mathbb{Z}^d} \tau(a^z)(1+|n|^2)^{-\frac{z}{2}}\\
             &= \tau(a^z)\Tr((1-\Delta)^{-\frac{z}{2}}).
    \end{align*}
    Since $a\geq 0,$ the function
    \[
        z\mapsto \tau(a^z)
    \]
    is holomorphic in the half-plane $\{\re(z)> 0\}.$ Applying Lemma \ref{zeta_function_lemma}, it follows that the function
    \[
        \Tr(\rho(a)^z(1-\Delta)^{-\frac{z}{2}})-\frac{\Vol(S^{d-1})\tau(a^d)}{z-d}
    \]
    has analytic continuation to the half-plane $\re(z)>0.$
\end{proof}
Lemma \ref{torus_analyticity_check} completes the verification of Condition \ref{zeta_condition} for the spectral triple $(C^\infty(\Circ^d_\theta)\otimes 1_{\Cplx^N},L_2(\Circ^d_\theta)\otimes \Cplx^{N},D).$
Indeed, for $0\leq a \in C^\infty(\Circ^d_\theta),$ the function
\[
    \Tr(\rho(a)^z(1-\Delta)^{-\frac{z}{2}})-\frac{\Vol(S^{d-1})\tau(a^d)}{z-d},\quad \re(z)>0
\]
is analytic, and hence is continuous on the smaller half-plane $\re(z)\geq d.$

Since the spectral triple $(C^\infty(\Circ^d_\theta)\otimes 1_{\Cplx^N},L_2(\Circ^d_\theta)\otimes \Cplx^{N},D)$ obeys the assumptions of Theorem \ref{main_spectral_triple_theorem} with $p=d,$ it follows immediately from the cited theorem that for all $0\leq a \in C(\Circ^d_\theta)$ the limit
\[
    \lim_{t\to\infty} t\mu(t,(1-\Delta)^{-\frac{d}{4}}\rho(a)(1-\Delta)^{-\frac{d}{4}})
\]
exists. The following theorem refines this statement by providing an explicit formula for the limit in terms of the trace $\tau.$
\begin{theorem}\label{strong_cif_nc_torus}
    For all $0\leq b \in C(\Circ^d_\theta),$ we have
    \[
        \lim_{t\to\infty} t\mu(t,(1-\Delta)^{-\frac{d}{4}}\rho(b)(1-\Delta)^{-\frac{d}{4}}) = \frac{\Vol(S^{d-1})}{d}\tau(b).
    \]
\end{theorem}
\begin{proof} 
    Applying Proposition \ref{abstract_cif|}, for all $0\leq a \in C^\infty(\Circ^d_\theta)$ we have
    \[
        \lim_{t\to\infty} t\mu(t,(1-\Delta)^{-\frac{d}{4}}\rho(a)^{2d}(1-\Delta)^{-\frac{d}{4}}) = \frac{1}{d}\lim_{z\downarrow d} (z-d)\Tr(\rho(a)^{2z}(1-\Delta)^{-\frac{z}{2}}).
    \]
    The above limit as $z\downarrow d$ has already been computed in Lemma \ref{torus_analyticity_check}, and we obtain
    \[
        \lim_{t\to\infty} t\mu(t,(1-\Delta)^{-\frac{d}{4}}\rho(a)^{2d}(1-\Delta)^{-\frac{d}{4}}) = \frac{\Vol(S^{d-1})}{d}\tau(a^{2d}),\quad a \in C^\infty(\Circ^d_\theta).
    \]
    Given $b\in C(\Circ^d_\theta),$ there exists a sequence $\{b_n\}_{n\geq 0}$ in $C^\infty(\Circ^d_\theta)$ such that $b_n\to b$ in the uniform norm.
    For every $n\geq 0$ we have
    \[
        \lim_{t\to\infty} t\mu(t,(1-\Delta)^{-\frac{d}{4}}\rho(b_n)^{2d}(1-\Delta)^{-\frac{d}{4}}) = \frac{\Vol(S^{d-1})}{d}\tau(b_n^{2d}).
    \]
    It follows from Lemma \ref{bs_perturbation_lemma_sequential} that
    \[
        \lim_{t\to\infty} t\mu(t,(1-\Delta)^{-\frac{d}{4}}\rho(b)^{2d}(1-\Delta)^{-\frac{d}{4}}) = \frac{\Vol(S^{d-1})}{d}\tau(b^{2d}).
    \]
    Replacing $b$ with $b^{\frac{1}{2d}}$ completes the proof.    
\end{proof}

\section{Application to semiclassical Weyl laws}
We will apply the following form of the Birman--Schwinger principle, adapted from \cite{MP2021}.
Recall that the \emph{essential spectrum} of a closed linear operator $T$ on a Hilbert space is that part of the spectrum which does not consist of isolated eigenvalues
of finite multiplicity \cite[pp.~236]{Reed-Simon-I}, \cite[Definition 8.3]{Schmudgen2012}. If $T:\mathrm{dom}(T)\to H$ is a self-adjoint positive unbounded linear operator with compact resolvent, i.e. $(1+T)^{-1}\in \Kc(H)$, then 
$T$ has no essential spectrum. If $T$ has compact resolvent and $V$ is a linear operator, then $T+V$ also has compact resolvent and it follows that for all self-adjoint bounded operators $V$, the sum $T+V$ has no essential spectrum. 

Given that $T+V$ has no essential spectrum and is lower bounded, the quantity
\[
    N(\lambda,T+V) = \Tr(\chi_{(-\infty,\lambda)}(T+V)),\quad \lambda \in \Rl
\]
is finite. Equivalently, $N(\lambda,T+V)$ is the number of eigenvalues (counting multiplicity) of the operator $T+V$ which are less than $\lambda \in \Rl.$ The Birman-Schwinger principle asserts
that if $\lambda <0$ then
\[
    N(\lambda,T+V) = \Tr(\chi_{(1,\infty)}(-(T-\lambda)^{-\frac12}V(T-\lambda)^{-\frac12})).
\]
That is, the number of eigenvalues less than $\lambda$ of $T+V$ is equal to the number of eigenvalues exceeding $1$ of the operator $-(T-\lambda)^{-\frac12}V(T-\lambda)^{-\frac12},$ \cite[Lemma 1.4]{BS1989b}, \cite[Lemma 7.1]{Simon2003}, \cite[Theorem 7.9.4]{SimonCourse4}.

The following theorem is not likely to be novel, but we have been unable to find an adequate reference.
\begin{theorem}\label{bs_principle}
        Let $T$ be a self-adjoint positive unbounded linear operator on a Hilbert space $H$ with compact resolvent. Let $V$ be a self-adjoint bounded linear operator.
        For all $q>0,$ we have
        \[  
            \lim_{h\downarrow 0}h^qN(0,hT+V) = \lim_{h\downarrow 0}h^q\Tr(\chi_{(h,\infty)}(-(1+T)^{-\frac12}V(1+T)^{-\frac12}))
        \]  
        if the limit on the right exists.
\end{theorem}

Besides using the Birman--Schwinger principle, our argument for Theorem \ref{bs_principle} is an application of the following elementary facts.
First, for all self-adjoint $T$ and $S$ and all $\alpha,\beta\in \Rl$ we have
\begin{equation}\label{subadditivity}
    \Tr(\chi_{(\alpha+\beta,\infty)}(T+S))\leq \Tr(\chi_{(\alpha,\infty)}(T))+\Tr(\chi_{(\beta,\infty)}(S)).
\end{equation}
In particular, if $S\geq 0$ then for all $t \in \Rl$ we have
\begin{equation}\label{monotonicity}
    \Tr(\chi_{(t,\infty)}(T))\leq \Tr(\chi_{(t,\infty)}(T+S)).
\end{equation}

Theorem \ref{bs_principle} is an immediate corollary of the following lemma, 
which is more general in that it does not assume the existence of any limit. 

\begin{lemma}  Let $T$ be a self-adjoint positive unbounded linear operator on a Hilbert space $H$ with compact resolvent. Let $V$ be a self-adjoint bounded linear operator. For all $q>0,$ we have
$$\limsup_{h\to0}h^qN(0,hT+V)\leq \limsup_{h\to0}h^q\Big(\Tr(\chi_{(h,\infty)}(-(1+T)^{-\frac12}V(1+T)^{-\frac12})))\Big),$$
$$\liminf_{h\to0}h^qN(0,hT+V)\geq \liminf_{h\to0}h^q\Big(\Tr(\chi_{(h,\infty)}(-(1+T)^{-\frac12}V(1+T)^{-\frac12})))\Big).$$    
\end{lemma}
\begin{proof} The Birman--Schwinger principle in the form above applied with $\lambda=h$ asserts that
    \[
        N(0,hT+V) = N(-h,hT+V-h) = \Tr(\chi_{(1,\infty)}(-(h+hT)^{-\frac12}(V-h)(h+hT)^{-\frac12})).
    \]
    Equivalently,
    \begin{equation}\label{exact_bs}
        N(0,hT+V) = \Tr(\chi_{(h,\infty)}(-(1+T)^{-\frac12}V(1+T)^{-\frac12}+h(1+T)^{-1})).
    \end{equation}
    For all $0 < \varepsilon < 1$, \eqref{subadditivity} implies that,
    \begin{align*}
        \Tr(\chi_{(h,\infty)}&(-(1+T)^{-\frac12}V(1+T)^{-\frac12}+h(1+T)^{-1}))\\
                             &\leq \Tr(\chi_{(h(1-\varepsilon),\infty)}(-(1+T)^{-\frac12}V(1+T)^{-\frac12}+h(1+T)^{-1})))\\
                             &\quad +\Tr(\chi_{(h \varepsilon,\infty)}(h(1+T)^{-1}))\\
                             &= \Tr(\chi_{(h(1-\varepsilon),\infty)}(-(1+T)^{-\frac12}V(1+T)^{-\frac12}))+\Tr(\chi_{(\varepsilon,\infty)}((1+T)^{-1})).
    \end{align*}

Thus,
\begin{align*}
 \limsup_{h\to0}h^qN(0,hT+V)&\leq\limsup_{h\to0}h^q\Big(\Tr(\chi_{(h(1-\varepsilon),\infty)}(-(1+T)^{-\frac12}V(1+T)^{-\frac12}))\\
                            &\quad +\Tr(\chi_{(\varepsilon,\infty)}((1+T)^{-1}))\Big)\\
                            &=\limsup_{h\to0}h^q\Big(\Tr(\chi_{(h(1-\varepsilon),\infty)}(-(1+T)^{-\frac12}V(1+T)^{-\frac12})))\Big)\\
                            &=(1-\epsilon)^{-q}\cdot \limsup_{h\to0}h^q\Big(\Tr(\chi_{(h,\infty)}(-(1+T)^{-\frac12}V(1+T)^{-\frac12})))\Big). 
\end{align*}
Since $\epsilon\in(0,1)$ is arbitrary, the first inequality follows

Since $h(T+1)^{-1}\geq 0$, \eqref{monotonicity} and \eqref{exact_bs} imply that
\begin{equation}\label{spectral_lower_bound}
N(0,hT+V) \geq \Tr(\chi_{(h,\infty)}(-(1+T)^{-\frac12}V(1+T)^{-\frac12})).
\end{equation}
The second inequality follows.
\end{proof}

Recall that if $S$ is a positive compact linear operator, then $\mu(S)$ can be described as the sequence of eigenvalues of $S$ arranged in decreasing order with multiplicities. Equivalently,
\begin{equation}\label{inverse_characterisation}
    \mu(t,S) = \inf\{s \geq 0\;:\;\Tr(\chi_{(s,\infty)}(S))\leq t\},\quad 0\leq S \in \Kc.
\end{equation}
The following formula is a standard. Its proof is an easy exercise.
\begin{lemma}\label{cut_to_positive_part}
    Let $S$ be a self-adjoint compact linear operator, with positive part $S_+.$ For all $q > 0,$ we have
    \[
        \lim_{h\downarrow 0} h^q\Tr(\chi_{(h,\infty)}(S)) = \lim_{t\to\infty} t\mu(t,S_+)^q.
    \]
    whenever either side exists.
\end{lemma}
%

Applying the preceding lemma to Theorem \ref{bs_principle} yields the following.
\begin{corollary}\label{bs_principle_in_terms_of_mu}
Let $T$ be a self-adjoint positive unbounded linear operator on a Hilbert space $H$ with compact resolvent. Let $V$ be a self-adjoint bounded linear operator.
	For all $q>0,$ we have
	\[  	\lim_{h\downarrow 0}h^qN(0,hT+V) = \lim_{t\to\infty}t\mu\Big(t,\Big((1+T)^{-\frac12}V(1+T)^{-\frac12}\Big)_-\Big)^q
	\]  
	if the limit on the right exists.

\end{corollary}

In the next lemma, we use the following fact: if $q > 1,$ and $S, T$ are self-adjoint linear operators such that
\[
    S-T \in \Lc_{q,\infty}
\]
then
\begin{equation}\label{positive_part_trick}
    S_+-T_+,\, S_--T_- \in \Lc_{q,\infty}
\end{equation}
where $S_{\pm}$ and $T_{\pm}$ are the positive and negative parts of $S$ and $T$ respectively, so that $T = T_+-T_-$ and $S = S_+-S_-.$

This is in fact a special case of the main result of \cite{PS2011} (compare the similar assertion \eqref{ps_acta_inequality}). To see this,
note that the main result of \cite{PS2011} implies that
\[
    f(S)-f(T) \in \Lc_{q,\infty}
\]
for every Lipschitz continuous function $f$. Since $f(t) = t_+$ and $f(t) = t_-$ are Lipschitz functions, \eqref{positive_part_trick} follows.

\begin{lemma}\label{last section commutator lemma}
Let $T$ be a self-adjoint positive unbounded linear operator, and let $V$ be a self-adjoint bounded linear operator such that
\begin{enumerate}[{\rm (i)}]
\item $(1+T)^{-\frac12}\in\mathcal{L}_{p,\infty};$
\item $[(1+T)^{-\frac12},V_{\pm}^{\frac12}] \in (\mathcal{L}_{p,\infty})_0;$
\end{enumerate}    
where $p>2.$ It follows that
$$((1+T)^{-\frac12}V(1+T)^{-\frac12})_- - V_{-}^{\frac12}(1+T)^{-1}V_{-}^{\frac12} \in (\mathcal{L}_{\frac{p}{2},\infty})_0.$$
\end{lemma}
\begin{proof} We write
$$-(1+T)^{-\frac12}V(1+T)^{-\frac12} = (1+T)^{-\frac12}(V_--V_+)(1+T)^{-\frac12}.$$
Hence,
$$(1+T)^{-\frac12}V_{\pm}(1+T)^{-\frac12}-V_{\pm}^{\frac12}(1+T)^{-1}V_{\pm}^{\frac12}=$$   
$$=[(1+T)^{-\frac12},V_{\pm}^{\frac12}]\cdot V_{\pm}^{\frac12}(1+T)^{-\frac12}-V_{\pm}^{\frac12}(1+T)^{-\frac12}\cdot[(1+T)^{-\frac12},V_{\pm}^{\frac12}].$$
By assumption, we have
$$(1+T)^{-\frac12}V_{\pm}(1+T)^{-\frac12}-V_{\pm}^{\frac12}(1+T)^{-1}V_{\pm}^{\frac12}\in(\mathcal{L}_{\frac{p}{2},\infty})_0.$$   
Consequently,
$$-(1+T)^{-\frac12}V(1+T)^{-\frac12}-V_{-}^{\frac12}(1+T)^{-1}V_{-}^{\frac12}+V_+^{\frac12}(1+T)^{-1}V_+^{\frac12}\in(\mathcal{L}_{\frac{p}{2},\infty})_0.$$    
Since $p>2,$ we may apply \eqref{positive_part_trick} with $q = \frac{p}{2}$ yielding
$$\Big(-(1+T)^{-\frac12}V(1+T)^{-\frac12}\Big)_+-\Big(V_{-}^{\frac12}(1+T)^{-1}V_{-}^{\frac12}-V_+^{\frac12}(1+T)^{-1}V_+^{\frac12}\Big)_+\in(\mathcal{L}_{\frac{p}{2},\infty})_0.$$    
Obviously,
$$\Big(V_{-}^{\frac12}(1+T)^{-1}V_{-}^{\frac12}-V_+^{\frac12}(1+T)^{-1}V_+^{\frac12}\Big)_+= V_{-}^{\frac12}(1+T)^{-1}V_{-}^{\frac12}.$$
Therefore,
$$\Big(-(1+T)^{-\frac12}V(1+T)^{-\frac12}\Big)_+-V_{-}^{\frac12}(1+T)^{-1}V_{-}^{\frac12}\in(\mathcal{L}_{\frac{p}{2},\infty})_0$$    
and the assertion follows.
\end{proof}

The next corollary follows immediately from Corollary \ref{simplified_bs_perturbation_lemma}, Lemma \ref{last section commutator lemma} and Corollary \ref{bs_principle_in_terms_of_mu}.
\begin{corollary}\label{from 55 and 54} Let $T$ be a self-adjoint positive unbounded linear operator on a Hilbert space $H$ with compact resolvent. Let $V$ be a self-adjoint bounded linear operator. Let $p>2$ and suppose the assumptions in Lemma \ref{last section commutator lemma} hold. We have
	\[  	\lim_{h\downarrow 0}h^{\frac{p}{2}}N(0,hT+V) = \lim_{t\to\infty}t\mu\Big(t,V_{-}^{\frac12}(1+T)^{-1}V_{-}^{\frac12}\Big)^{\frac{p}{2}}
	\]   
	if the limit on the right exists.
\end{corollary}

We now replace $\mu(V_-^{\frac12}(1+T)^{-1}V_-^{\frac12})^{\frac{p}{2}}$ with $\mu(V_-^{\frac{p}{4}}(1+T)^{-\frac{p}{2}}V_-^{\frac{p}{4}}).$
\begin{corollary}\label{very last corollary} Let $T$ be a self-adjoint positive unbounded linear operator on a Hilbert space $H$ with compact resolvent. Let $V$ be a self-adjoint bounded linear operator. Let $p>2$ and suppose the assumptions in Lemma \ref{last section commutator lemma} hold. We have
	\[  	\lim_{h\downarrow 0}h^{\frac{p}{2}}N(0,hT+V) = \lim_{t\to\infty}t\mu\Big(t,V_{-}^{\frac{p}{4}}(1+T)^{-\frac{p}{2}}V_{-}^{\frac{p}{4}}\Big)\]  
	if the limit on the right exists.
\end{corollary}
\begin{proof} Applying Lemma \ref{from_csz_lemma} with $A=V_-,$ $B=(1+T)^{-1}$ and $\frac{p}{2}$ instead of $p,$ we obtain that
$$\Big(V_{-}^{\frac12}(1+T)^{-1}V_{-}^{\frac12}\Big)^{\frac{p}{2}}-V_{-}^{\frac{p}{4}}(1+T)^{-\frac{p}{2}}V_{-}^{\frac{p}{4}}\in(\mathcal{L}_{1,\infty})_0.$$	
Since the limit
\[
    \lim_{t\to\infty}t\mu\Big(t,V_{-}^{\frac{p}{4}}(1+T)^{-\frac{p}{2}}V_{-}^{\frac{p}{4}}\Big)
\]
exists, Corollary \ref{simplified_bs_perturbation_lemma} implies that
$$\mu\Big(t,V_{-}^{\frac12}(1+T)^{-1}V_{-}^{\frac12}\Big)^{\frac{p}{2}}-\mu\Big(t,V_{-}^{\frac{p}{4}}(1+T)^{-\frac{p}{2}}V_{-}^{\frac{p}{4}}\Big)=o(t^{-1}),\quad t\to\infty.$$
The assertion follows from Corollary \ref{from 55 and 54}.
\end{proof}

\begin{lemma}\label{second last lemma} Suppose we are in the conditions of Theorem \ref{main_spectral_triple_theorem}. For every $V=V^{\ast}\in \overline{\Ac},$ there exists a limit
$$\lim_{t\to\infty}t\mu(t,(1+D^2)^{-\frac{p}{4}}V_-^{\frac{p}{2}}(1+D^2)^{-\frac{p}{4}}).$$	
\end{lemma}
\begin{proof} 
Since the function $f:x\to x_+^{\frac{p}{2}},$ $x\in\mathbb{R},$ is continuous and $\overline{\Ac}$ is a $C^*$-algebra, it follows that $V_-^{\frac{p}{2}} \in \overline{\Ac}.$ The result is now a special case of Proposition \ref{abstract_cif}.

\end{proof}

With the preceding abstract preliminaries in place, we now proceed to the proof of Theorem \ref{scwl_spectral_triple_thm}.
\begin{proof}[Proof of Theorem \ref{scwl_spectral_triple_thm}]
Initially take $\lambda=0.$ We have that $V_-^{\frac{1}{2}}\in \overline{\Ac},$ and hence Lemma \ref{section 5 cts lemma} implies that
\[
    [(1+D^2)^{-\frac{1}{2}},V_-^{\frac{1}{2}}] \in (\Lc_{p,\infty})_0.
\]
It follows that the assumptions in Lemma \ref{last section commutator lemma} hold for $T=D^2$ and $V.$ Thus, Corollary \ref{very last corollary} asserts that
\[  	\lim_{h\downarrow 0}h^{\frac{p}{2}}N(0,hT+V) = \lim_{t\to\infty}t\mu\Big(t,V_{-}^{\frac{p}{4}}(1+T)^{-\frac{p}{2}}V_{-}^{\frac{p}{4}}\Big)\]  
if the limit on the right exists. However, 
$$\mu\Big(V_{-}^{\frac{p}{4}}(1+T)^{-\frac{p}{2}}V_{-}^{\frac{p}{4}}\Big)=\mu\Big((1+D^2)^{-\frac{p}{4}}V_-^{\frac{p}{2}}(1+D^2)^{-\frac{p}{4}}\Big)$$
and thus the existence of the limit is now guaranteed by Lemma \ref{second last lemma}.

Replacing $V$ with $V-\lambda$ completes the proof in the general case.
\end{proof}

As illustration, we give a semiclassical Weyl law for the noncommutative torus.
\begin{corollary}\label{scwl_torus_continuous}
    Let $d>2.$ For all $V=V^* \in C(\Circ^d_\theta)$ we have
    \[
        \lim_{h\downarrow 0} h^{d}N(\lambda,-h^2\Delta+\rho(V)) = \frac{\Vol(S^{d-1})}{d}\tau((V-\lambda)_-^{\frac{d}{2}}),\quad \lambda\in \Rl.
    \]
\end{corollary}
\begin{proof}
    Theorems \ref{scwl_spectral_triple_thm} and \ref{strong_cif_nc_torus}  (with $b = V_-^{\frac{d}{2}}$) imply the existence and coincidence of the limits
    \[
        \lim_{h\downarrow 0} h^{d}N(\lambda,-h^2\Delta+\rho(V)) = \lim_{t\to\infty} t\mu(t,(1-\Delta)^{-\frac{d}{4}}(\rho(V)-\lambda)_-^{\frac{d}{2}}(1-\Delta)^{-\frac{d}{4}}),\quad \lambda\in \Rl.
    \]
    From Theorem \ref{strong_cif_nc_torus}, the latter limit is equal to $\frac{\Vol(S^{d-1})}{d}\tau((V-\lambda)_-^{\frac{d}{2}}).$
\end{proof}

The assumption that $V \in C(\Circ^d_\theta)$ in Corollary \ref{scwl_torus_continuous} is far stronger than necessary. 
It was proved in \cite{MP2021} for $d>2$ that if $V = V^*\in L_{\frac{d}{2}}(\Circ^d_\theta),$ then there exists a self-adjoint operator $-h^2\Delta+\rho(V)$ defined
as a sum of quadratic forms. Here, $L_{\frac{d}{2}}(\Circ^d_\theta)$ is an $L_p$-space on the noncommutative torus, we refer to \cite{MP2021} for details.
We may extend Corollary \ref{scwl_torus_continuous} to a much wider class of $V$ by following the arguments of Simon in the classical case \cite{Simon1976} and applying a Cwikel--Lieb--Rozenblum inequality from \cite{MP2021}.
We apply the following inequality (which is a version of \eqref{subadditivity})
\begin{equation}\label{form_subadditivity}
    N(0,T+S)\leq N(0,T)+N(0,S).
\end{equation}
Here, $T$ and $S$ are potentially unbounded self-adjoint operators and $T+S$ may be defined in the sense of quadratic forms.
\begin{corollary}\label{conjecture_confirmation}
    Let $d>2.$ For all $V = V^* \in L_{\frac{d}{2}}(\Circ^d_\theta),$ we have
    \[
        \lim_{h\downarrow 0} h^dN(\lambda,-h^2\Delta+\rho(V)) = \frac{\Vol(S^{d-1})}{d}\tau((V-\lambda)_-^{\frac{d}{2}}),\quad \lambda\in \Rl.
    \]
\end{corollary}
\begin{proof}
    Replacing $V$ by $V-\lambda$ if necessary, it suffices to take $\lambda=0.$
    The Cwikel--Lieb--Rozenblum inequality from \cite[Corollary 8.7]{MP2021} implies that if $V =V^*\in L_{\frac{d}{2}}(\Circ^d_\theta),$ then
    \begin{equation}\label{clr_estimate}
        \limsup_{h\downarrow 0} h^dN(0,-h^2\Delta+\rho(V)) \leq c_{d}\|V_-\|_{\frac{d}{2}}^{\frac{d}{2}}.
    \end{equation}
    where $c_{d}>0$ is a constant. Since $C(\Circ^d_\theta)$ is dense in $L_{\frac{d}{2}}(\Circ^d_\theta),$ for every $0 < \varepsilon < 1,$ we can select $V_{\varepsilon}=V_{\varepsilon}^*\in C(\Circ^d_\theta)$ be such that $\|V-V_{\varepsilon}\|_{\frac{d}{2}} \leq \varepsilon^{2}.$
    Note that this implies
    \[
        \lim_{\varepsilon\to 0} \tau((V_{\varepsilon})_-^{\frac{d}{2}}) = \tau(V_-^{\frac{d}{2}}).
    \]
    
    Using the decomposition 
    \[
        -h^2\Delta+\rho(V) = (-h^2(1-\varepsilon)\Delta+\rho(V_{\varepsilon}))+(-h^2\varepsilon\Delta+\rho(V-V_{\varepsilon}))
    \]
    and \eqref{form_subadditivity} yields
    \[
        N(0,-h^2\Delta+\rho(V)) \leq N(0,-h^2(1-\varepsilon)\Delta+\rho(V_{\varepsilon}))+N(0,-h^2\varepsilon\Delta+\rho(V-V_{\varepsilon})).
    \]
    Applying Corollary \ref{scwl_torus_continuous} to the first term and \eqref{clr_estimate} to the second term gives the inequality
    \begin{align*}
        \limsup_{h\downarrow 0} h^dN(0,-h^2\Delta+\rho(V)) &\leq (1-\varepsilon)^{-\frac{d}{2}}\frac{\Vol(S^{d-1})}{d}\tau((V_{\varepsilon})_-^{\frac{d}{2}})+\varepsilon^{-{\frac{d}{2}}}c_d\|(V-V_\varepsilon)_-\|_{\frac{d}{2}}^{\frac{d}{2}}\\
                                                           &\leq (1-\varepsilon)^{-\frac{d}{2}}\frac{\Vol(S^{d-1})}{d}\tau((V_{\varepsilon})_-^{\frac{d}{2}})+c_d\varepsilon^{\frac{d}{2}}.
    \end{align*}
    Since $\varepsilon>0$ is arbitrary, we may take $\varepsilon\to 0$ to arrive at
    \[
        \limsup_{h\downarrow 0} h^dN(0,-h^2\Delta+\rho(V)) \leq \frac{\Vol(S^{d-1})}{d}\tau(V_-^{\frac{d}{2}}).
    \]
    Following a similar argument yields
    \[
        \liminf_{h\downarrow 0} h^dN(0,-h^2\Delta+\rho(V)) \geq \frac{\Vol(S^{d-1})}{d}\tau(V_-^{\frac{d}{2}})
    \]
    which completes the proof.
\end{proof}
This provides an affirmative answer to the Conjecture 8.8 of \cite{MP2021} in the special case $d>2$ and $p = \frac{d}{2}.$

\end{document}